\documentclass{article}
\usepackage{graphicx} 
\usepackage{amsmath}
\usepackage{amssymb}
\usepackage{subfigure}
\usepackage{caption}
\usepackage{amsthm}%
\usepackage{tikz-cd}
\usepackage{comment}
\usepackage{geometry}

\newtheorem{theorem}{Theorem}%  meant for continuous numbers

\newtheorem{proposition}[theorem]{Proposition}% 
\newtheorem{lemma}[theorem]{Lemma}
\newtheorem{corollary}[theorem]{Corollary}% 

\numberwithin{theorem}{section}

\title{Symmetries of $(3, 6)$-Fullerenes}

\author{Linda Green, Yadunand Sreelesh, Vibhu Gomatam}

\begin{document}

\maketitle

\abstract{\noindent A $(3, 6)$-fullerene is a cubic planar graph whose faces all have 3 or 6 sides. We give an exact enumeration of $(3, 6)$-fullerenes with $V$ vertices and each of five possible symmetry types: $\star 332$, $332$, $2 \star 2$, $\star 222$,  and $222$, in orbifold notation. We use this enumeration, together with some constructions, to resolve three  conjectures about the existence of $(3, 6)$-fullerenes with $2 \star 2$, $\star 222$, and $222$ symmetry for given numbers of vertices.}

%% Keywords

\vspace{0.5 cm}

%%  The body

\section{Introduction}
\label{sec-intro}
A \emph{$(3, 6)$-fullerene} is a finite, connected, cubic, planar graph whose faces all have three or six sides, with a fixed planar embedding.  We will say that two  $(3, 6)$-fullerenes are equivalent if they are not only isomorphic as graphs but if there is also an orientation-preserving homeomorphism of the plane that is a graph isomorphism. Thus, left-handed and right-handed $(3,6)$-fullerenes are considered distinct. We refer to faces with three sides as  ``triangles'' and faces with six sides as ``hexagons'', even though these faces may not have straight edges and may be unbounded. We will sometimes consider a $(3,6)$-fullerene to be embedded in the sphere instead of the plane, using the one-point compactification of the plane.

It is known that $(3, 6)$-fullerenes must have one of five possible symmetry types. In orbifold notation, these symmetry types are $\star 332$, $332$, $2\star 2$, $\star 222$,  and $222$  \cite{deza2005zigzag}. In Sch\"onflies notation, these symmetry types are $T_d$, $T$, $D_{2d}$, $D_{2h}$, and $D_2$, respectively. This paper enumerates the $(3,6)$-fullerenes with each symmetry type for every possible number of vertices. 

Section~\ref{sec-background} provides background information about ``signatures'': the signature of a $(3,6)$-fullerene is a triple of numbers that describes the position of its triangles among its hexagons. Section~\ref{sec-symmDefs} connects alternative ideas of symmetry, with proofs postponed to Section~\ref{sec-symmProof}. Section~\ref{sec-sigAndSymm} relates signatures and symmetry type. Section~\ref{sec-countsBySymmType} uses signatures to enumerate the $(3,6)$-fullerenes with each symmetry type. Section~\ref{sec-conj1And2} proves three conjectures from \cite{deza2005zigzag} about the existence of $(3, 6)$-fullerenes with $\star 222$, $2\star 2$, and $222$ symmetry for given numbers of vertices.

This paper builds on the enumeration of (3, 6)-fullerenes and the count of those with 3-fold rotational symmetry, mirror symmetry, and both kinds of symmetry, established in \cite{green2025enumerate}. Here we go further in two ways. First, we refine that count to give the exact number of $(3, 6)$-fullerenes of each of the five symmetry types, for every possible number of vertices (Section 5), which requires the new signature-to-symmetry-type dictionary of Section 4. Second, we use this refined enumeration, together with explicit constructions, to resolve the existence conjectures of Deza and Dutour (\cite{deza2005zigzag}, Conjecture 5.7) concerning which vertex counts admit $(3, 6)$-fullerenes with $2\star 2$, $\star 222$,  and $222$ symmetry (Section 6). The corresondence between the spherical symmetry type of a $(3, 6)$-fullerene and the wallpaper symmetry type of 
its hexagonal tiling cover is also new to this paper (Propositions 3.1 and 3.2, proved in Section 7). 

The results in this paper can be applied to polyhedra whose faces are all triangles and hexagons and have three faces around each vertex, but are not necessarily
convex. Every $(3, 6)$-fullerene can be realized as such a polyhedron \cite{green2024polyhedra}.

\section{Background} 
\label{sec-background}

Consider a regular hexagonal grid on the plane and let $Z$ be the vertices of a superimposed parallelogram grid, where each vertex of the parallelogram grid lies in the center of a hexagon. Let $\Gamma$ be the group of isometries of the plane generated by $180^\circ$ rotations around the points of $Z$. These isometries preserve both $Z$ and the hexagonal grid. Let $Q = \Gamma \backslash \mathbb{R}^2$ be the orbifold quotient space of the plane under this group action, where elements of the quotient space are the orbits of points in the plane under the action of $\Gamma$. The hexagons that contain the vertices of the parallelogram grid at their centers are called \emph{special hexagons}, and they project to triangles with cone points in the quotient orbifold. The hexagons that are not special hexagons project to hexagons in the quotient orbifold. The entire plane quotients to a surface that is homeomorphic to the sphere \cite{green2024polyhedra}. The vertices and edges of the hexagonal grid quotient to a cubic, planar graph whose faces are all triangles and hexagons, embedded in the sphere via this homeomorphism.  Previous work \cite{green2024polyhedra} shows that every $(3,6)$-fullerene can be represented this way, as the quotient of a hexagonal tiling of the plane under a group of isometries $\Gamma$ generated by $180^\circ$ rotations around the vertices of a superimposed parallelogram grid. See Figure~\ref{fig-upstairsDownstairs} for an example. The quotient $Q$ inherits a natural metric from $\mathbb{R}^2$, where the distance between two orbits is given by $d(\Gamma.x, \Gamma.y) = \inf_{\gamma \in \Gamma} d(x, \gamma(y))$ \cite{bridson2013metric}. We will refer to $Q$ with this metric as a \emph{quotient representation of a $(3,6)$-fullerene}.
%NOTE - To see that the parallelogram grid can be drawn so that some edges of parallelograms are vertical, write the original parallelogram vectors as linear combinations of vertical and SW->NE vectors and then taking a linear combination of these to get a vertical vector.
 Instead of working ``downstairs" with the $(3, 6)$-fullerene directly, many symmetry arguments become simpler by working ``upstairs" in the hexagonal tiling of the plane that covers the quotient representation. 

Given a hexagonal tiling with a superimposed parallelogram grid, the hexagonal tiling can be oriented so that the hexagons form vertical columns as in Figure~\ref{fig-upstairsDownstairs} and the parallelogram grid can be drawn so that the edges of the parallelograms include vertical segments. A vertical column containing special hexagons in the hexagonal tiling upstairs is called a \emph{spine column}. It projects downstairs to a strip of hexagons capped by two triangles, called a \emph{spine}. We say that two spine columns are \emph{consecutive} if there are no other spine columns between them. The vertical columns between pairs of consecutive spine columns are called \emph{belt} columns. They project downstairs to bands of hexagons called \emph{belts}. 

\begin{figure}

\centering

\includegraphics[height = 6 cm]{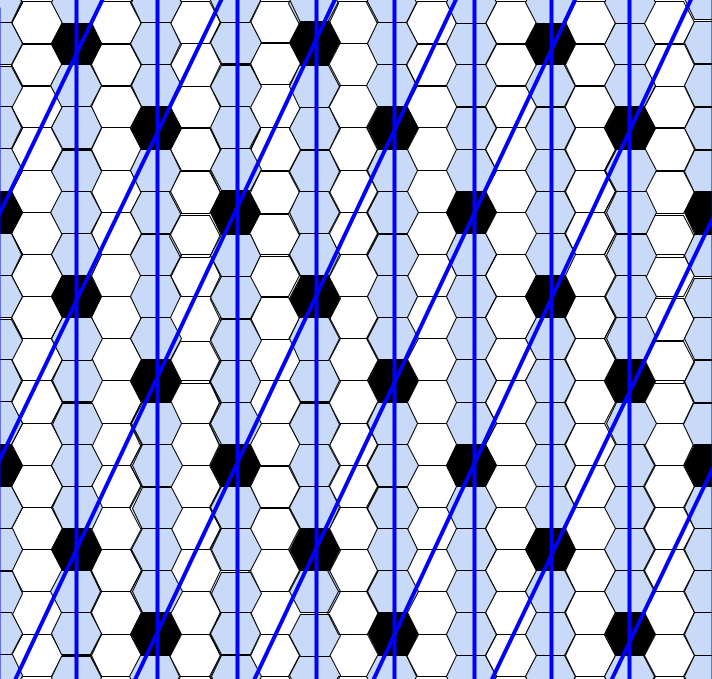}
\includegraphics[height = 6 cm]{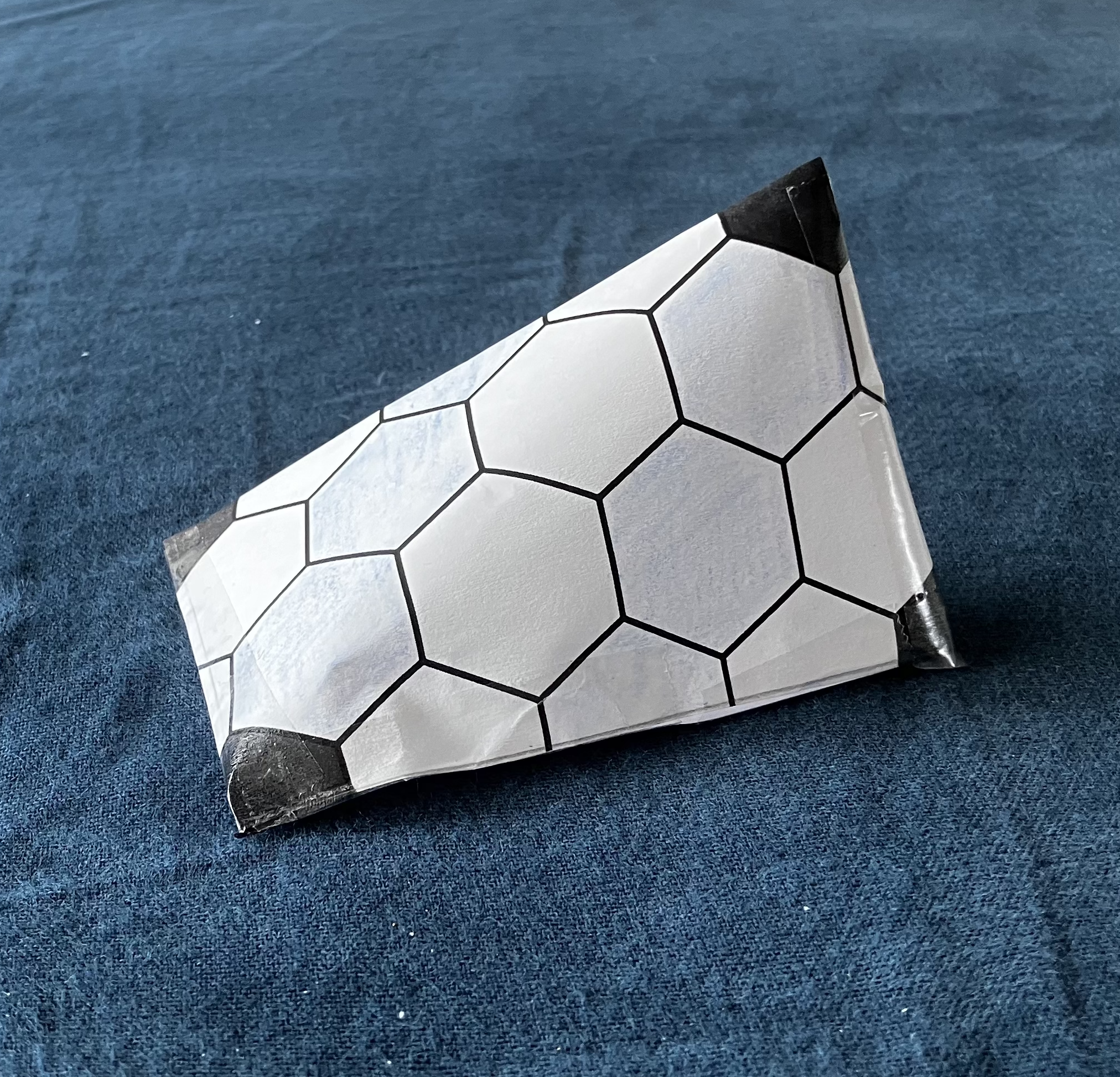}
\caption{A hexagonal tiling with a superimposed parallelogram grid produces a $(3, 6)$-fullerene as its quotient under the group generated by $180^\circ$ rotations around the vertices of the parallelogram grid. Special hexagons are shaded black and spines are shaded light blue.}
\label{fig-upstairsDownstairs}
\end{figure}

As detailed in \cite{green2024polyhedra}, every $(3, 6)$-fullerene can also be fully described by a \emph{signature}, which is  triple of non-negative integers $(s, b, f)$, where $f$ is defined modulo $s+1$. The first number, $s$, gives the number of hexagons that lie in the vertical column between two special hexagons, which is also called the \emph{length of a spine}. The second number, $b$, gives the number of belt columns between two consecutive spine columns.  
The third number, $f$, is called the \emph{offset}. It can be found by translating a special hexagon in an approximately southwest to northeast direction along other hexagons, through $b+1$ vertical columns until the next spine column is reached. The offset $f$ is the number of hexagons below a special hexagon that the translated hexagon lands. For example, in Figure~\ref{fig-upstairsDownstairs}, the offset is 3 and the signature is $(5, 1, 3)$

Signatures are not unique: each $(3, 6)$-fullerene has three possible signatures corresponding to three orientations of the plane: the original orientation, the plane rotated by $60^\circ$ counterclockwise, and the plane rotated by $120^\circ$ counterclockwise. (The plane rotated by $180^\circ$ counterclockwise is identical to the original.) For example, in Figure~\ref{fig-upstairsDownstairs}, the three signatures are $(5, 1, 3)$, $(3, 2, 3)$, and $(11, 0, 2)$. For any $(3, 6)$-fullereine, the three equivalent signatures $(s_1, b_1, f_1)$, $(s_2, b_2, f_2)$, and $(s_3, b_3, f_3)$ are related by six equations given in \cite{green2024polyhedra}, and restated here for convenience. These equations are used in the proof of the existence of $(3, 6)$-fullerenes with specified symmetry type and number of vertices in Section~\ref{sec-conj1And2}.

\begin{equation}
\label{eqn-s2}
s_2 = j_2(b_1 + 1) - 1
\end{equation}
where $j_2$ is the order of $f_1$ in $\mathbf{Z}_{s_1 + 1}$, that is, the smallest positive integer such that $j_2 \cdot f_1 \equiv 0 \pmod {s_1 +1 }$,

\begin{equation}
\label{eqn-b2}
b_2 = \dfrac{h - 2s_2}{2s_2 + 2}
\end{equation}
 where $h$, the number of hexagons, is given by $h = 2s_1  b_1 + 2s_1 + 2b_1$,

\begin{equation}
\label{eqn-f2}
f_2 \equiv  - p_2(b_1 + 1) - (b_2 +1)\pmod {s_2 + 1}
\end{equation}
where $p_2$ is the smallest positive integer such that  $p_2 \cdot f_1 \equiv (b_2 + 1) \pmod{s_1 + 1}$,
\begin{equation}
\label{eqn-s3}
s_3 = j_3(b_1 + 1) - 1
\end{equation} 
where $j_3$ is the order of $f_1 + b_1 + 1$ in $\mathbf{Z}_{s_1 + 1}$, that is, the smallest positive integer such that $j_3 \cdot(f_1 + b_1 + 1) \equiv 0 \pmod {s_1 + 1}$,
\begin{equation}  
\label{eqn-b3}
b_3 = \dfrac{h - 2s_3}{2s_3 + 2}
\end{equation}
 where $h$, the number of hexagons, is given by $h = 2s_1  b_1 + 2s_1 + 2b_1$,
 and
\begin{equation}
\label{eqn-f3}
f_3 \equiv  - p_3(b_1 + 1) \pmod{s_3 + 1}
\end{equation}
where $p_3$ is the smallest positive integer with $p_3 \cdot (f_1 + b_1 + 1) \equiv (b_3 + 1) \pmod{s_1 + 1}$.

In some cases, these equations generate \emph{coinciding signatures}: that is, all three signatures are exactly the same. This situation corresponds to the $(3,6)$-fullerene having \emph{3-fold rotational symmetry} \cite{green2025enumerate}, meaning that the $(3, 6)$-fullerene is equivalent to
an embedding on the sphere in which a 3-fold rotation of the sphere induces a graph
isomorphism. For example, the $(3, 6)$-fullerene with signature $(8, 2, 3)$ has coinciding signatures and 3-fold rotational symmetry. See Figure~\ref{fig-823}.

The number of vertices in a $(3, 6)$-fullerene is related to the signature via the equation
\begin{equation}
\label{eqn-vCount}
V = 4(s+1)(b+1),
\end{equation}
where $(s, b, f)$ is any of the three signatures \cite{green2024polyhedra}. It follows that the number of vertices is always a multiple of four. A standard Euler characteristic argument shows that the number of triangles in a $(3, 6)$-fullerene is always four \cite{grunbaum_motzkin_1963}, see also \cite{green2024polyhedra}.

All $(3,6)$-fullerenes are 3-connected with the exception of one family of 2-connected planar graphs, denoted by $G_n$ in Section 4 of \cite{deza2005zigzag} and referred to as godseyes in Section 2 of \cite{green2024polyhedra}. The planar embedding of a godseye is composed of
two triangles, surrounded by one or more nested pairs of hexagons that meet along
opposite sides, with a final pair of adjacent triangles on the outside. See Figure~\ref{fig-godseye}.

\begin{figure}[ht]
\begin{center}
\includegraphics[height = 4 cm]{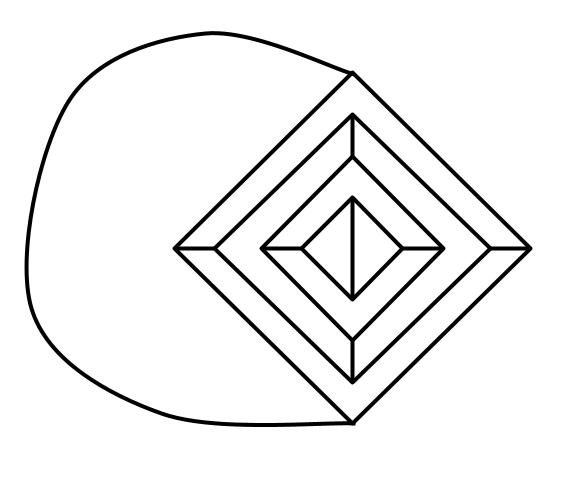}
\end{center}
\caption{A godseye is the only type of $(3, 6)$-fullerene that is not 3-connected.}
\label{fig-godseye}
\end{figure}

\section{Symmetry Groups}
\label{sec-symmDefs}

Let $F$ be the graph of a $(3, 6)$-fullerene. The graph $F$ can always be embedded in the sphere in an asymmetric way, but it may also be embedded in such a way that there are isometries of the sphere that preserve its vertices and edges.  For a fixed embedding in the sphere $f:F \to S^2$, let $Isom(S^2, F, f)$ be the group of isometries of the sphere that \emph{preserve} the vertices and edges of $f(F)$, meaning that vertices of $f(F)$ are taken to vertices of $f(F)$ and edges of $f(F)$ are taken to edges of $f(F)$. We will say that the embedding $f: F \to S^2$ is \emph{maximally symmetric} if there does not exist an embedding $f': F \to S^2$ such that $Isom(S^2, F, f)$ is a proper subgroup of $Isom(S^2, F, f')$. We will define the \emph{symmetry group} of a $(3,6)$-fullerene as the isometry group of a maximally symmetric embedding into the sphere. Proposition~\ref{prop-symmEquivSummary}  shows that this definition is well-defined, because the symmetry group of any maximally symmetric embedding of a $(3,6)$-fullerene is isomorphic to the isometry group of its quotient representation $Q$ described in Section~\ref{sec-background}.  

\begin{proposition}
    Suppose $Q$ is a quotient representation of a $(3,6)$-fullerene with edge and vertex graph $F$. Then $Isom(Q) \cong Isom(S^2, F, f_0)$, where $f_0: F \to S^2$ is any maximally symmetric embedding. 
    \label{prop-symmEquivSummary}
\end{proposition}

In the following proposition, the symmetry type of the $(3, 6)$-fullerene refers to the symmetry group of a maximally symmetric embedding of the $(3,6)$-fullerene, or equivalently, the symmetry group of the quotient representation $Q$. The symmetry type of the hexagonal tiling cover refers to the group of isometries of the plane that preserve the tiling by regular hexagons and also preserve the centers of special hexagons, i.e. the vertices of the superimposed parallelogram grid used to make the quotient $(3,6)$-fullerene $Q$. Spherical and wallpaper symmetry types are written in orbifold notation.

\begin{proposition}
\begin{enumerate}
   \item A (3,6)-fullerene has $\star 332$ symmetry if its hexagonal tiling cover has $\star 632$ symmetry.

    \item A (3,6)-fullerene has $332$ symmetry if its hexagonal tiling cover has $632$ symmetry.
    
   \item A (3,6)-fullerene has $2\star 2$ symmetry if its hexagonal tiling cover has $2 \star 22$ symmetry.

    \item A (3,6)-fullerene has $\star 222$ symmetry if its hexagonal tiling cover has $\star 2222$ symmetry.

    \item A (3,6)-fullerene has $222$ symmetry if its hexagonal tiling cover has $2222$ symmetry.

\end{enumerate}
\label{prop-upstairsDownstairsSymm}
\end{proposition}

The proofs of Propositions~\ref{prop-symmEquivSummary} and \ref{prop-upstairsDownstairsSymm} are given in Section~\ref{sec-symmProof}.

\section{Symmetry and signatures}
\label{sec-sigAndSymm}

This section establishes a correspondence between the signatures of $(3, 6)$-fullerene and their symmetry types.

\begin{lemma}
    Suppose a (3,6)-fullerene has signature $(s, b, \dfrac{s-b}{2})$ where $s$ and $b$ are both even or both odd. Suppose also that the (3,6)-fullerene has coinciding signatures. Then the (3,6)-fullerene has symmetry type $\star 332$. 
    \label{lem-star332}
\end{lemma}

\begin{figure}[ht]
\centering
\includegraphics[height = 8 cm]{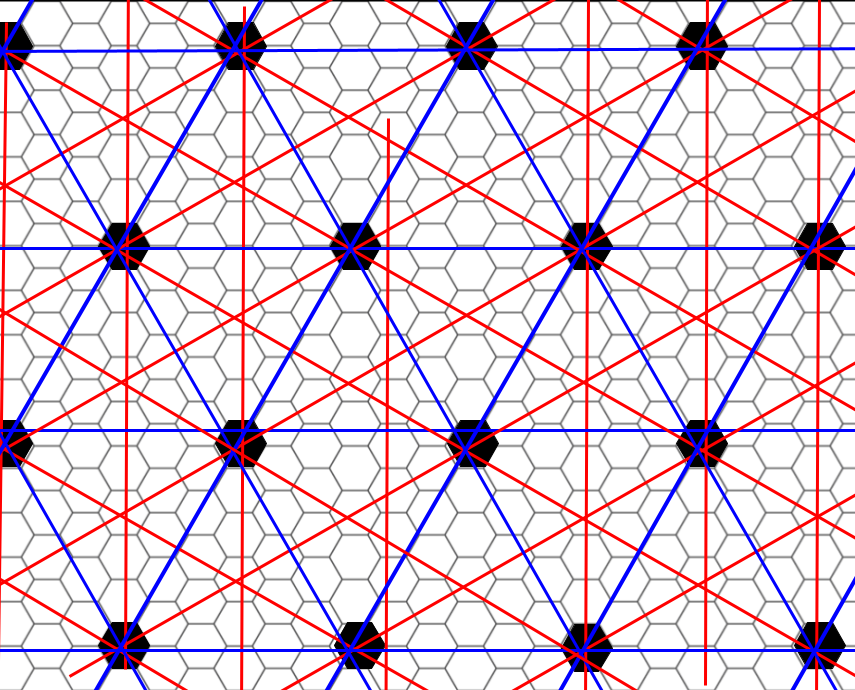}
    \caption {The $(3,6)$-fullerene with signature $(8, 2, 3)$ has $\star 632$ symmetry upstairs and $\star 332$ symmetry downstairs. Mirror lines are drawn.}
\label{fig-823}
\end{figure}

\begin{proof}
The offset of $\dfrac{s-b}{2} $ means that if a special hexagon, which we will call hexagon $A$, is shifted in the approximately SW to NE direction through $b+1$ columns of hexagons, it ends up $\dfrac{s-b}{2} \pmod{s+1}$ hexagons below a special hexagon. Call this second special hexagon $B$. Note that every column that hexagon $A$ is  shifted over in the SW to NE direction puts it a height of half a hexagon above its original horizontal position. Therefore, if hexagon $A$ ends up is $\dfrac{s-b}{2} \pmod{s+1}$ hexagons below a hexagon $B$ after shifting $b+1$ columns, it must have been at a height of $\dfrac{s-b}{2} + \dfrac{b+1}{2} \pmod{s+1}$ hexagons below hexagon $B$ to begin with. The number $\dfrac{s-b}{2} + \dfrac{b+1}{2}$ is equal to $\dfrac{s+1}{2}$, which means that hexagon $A$ lies precisely on the horizontal line halfway in between hexagon $B$ and the special hexagon below it. Similarly, every other special hexagon in a spine column lies on a horizontal line halfway between two special hexagons in the next consecutive spine column. It follows that the special hexagons are laid out on a grid as in Figures~\ref{fig-823} and \ref{fig-421}, so that each horizontal and vertical line that bisects special hexagons is a mirror line of the hexagonal grid containing special hexagons.

But by assumption, the (3,6)-fullerene has coinciding signatures, so by Proposition~3.2 of \cite{green2025enumerate}, in the hexagonal cover, rotation by $60^\circ$ around the center of any special hexagon takes special hexagons to special hexagons, and is therefore a symmetry of the hexagonal tiling cover. The hexagonal cover therefore has 6 mirror lines through the center of each special hexagon, and by the classification of wallpaper patterns \cite{conway2016symmetries}, it must have symmetry type $\star 632$, as in Figure~\ref{fig-823}. 

Therefore, by Proposition~\ref{prop-upstairsDownstairsSymm} the (3,6)-fullerene has symmetry type $\star 332$. 
\end{proof}

\begin{lemma}
    Suppose a (3,6)-fullerene has signature $(s, b, \dfrac{s-b}{2})$ where $s$ and $b$ are both even or both odd. Suppose also that the (3,6)-fullerene does not have coinciding signatures. Then the (3,6)-fullerene has symmetry type $2 \star 2$. 
    \label{lem-2star2}
\end{lemma}

\begin{proof}
    As in the proof of Lemma~\ref{lem-star332}, the hexagonal tiling cover must have horizontal mirror lines and vertical mirror lines through each special hexagon. There are no mirror lines that are at an angle and not horizontal or vertical, by the following argument. If there were any mirror lines at an angle, they would have to be at at $60^\circ$ angle from horizontal or vertical, in order to take hexagons to hexagons. The composition of a reflection through a mirror line at a $60^\circ$ angle to horizontal (or vertical) with the horizontal mirror (or vertical mirror) is a rotation by $120^\circ$, so there would be a $120^\circ$ rotation that takes special hexagons to special hexagons. By \cite{green2025enumerate}, this would imply that the (3,6)-fullerene has 3-fold rotational symmetry, a contradiction to assumptions. Therefore, the only mirror lines are horizontal and vertical. 

\begin{figure}[ht]
      \centering
\includegraphics[height = 8 cm]{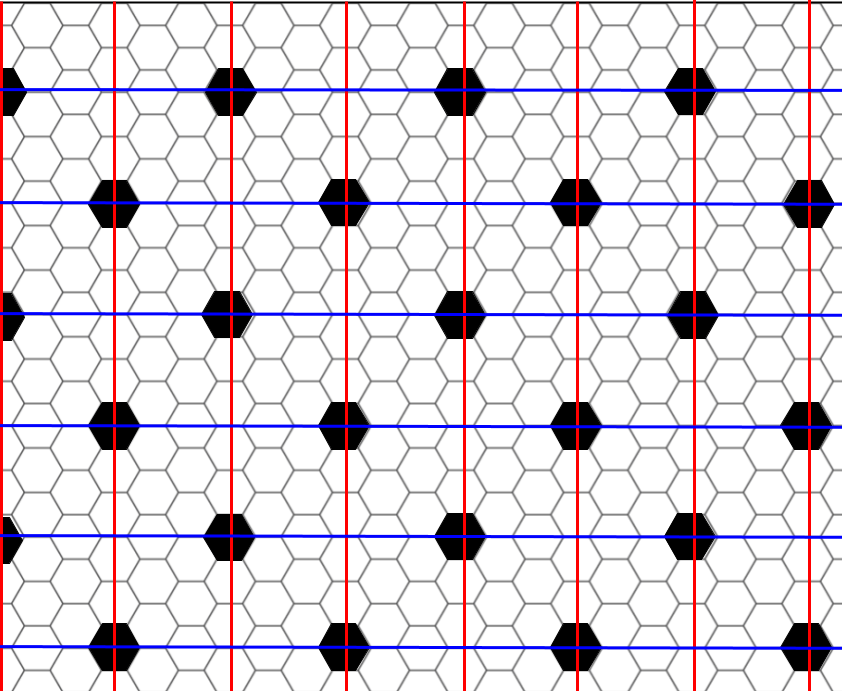}
    \caption {The $(3,6)$-fullerene with signature $(4, 2, 1)$ has $2\star 22$ symmetry upstairs and $2\star 2$ symmetry downstairs. Mirror lines are drawn.}
    \label{fig-421}
\end{figure}

    As in the proof of Lemma~\ref{lem-star332}, the special hexagons in two consecutive spine columns are not at the same height: the special hexagons in one vertical column lie precisely on the horizontal line halfway in between special hexagons in the next spine column. So there are no additional vertical mirrors besides the ones that bisect spines. For every spine column, any horizontal mirrors must either go through a special hexagon in that column or bisect a spine in that column. So there can be no additional horizontal mirrors besides the ones in Figure~\ref{fig-421}. Therefore, there are exactly two distinct types of points with intersecting mirrors that are not related by an isometry that takes special hexagons to special hexagons: one in the center of special hexagons and one halfway in between special hexagons. By the classification of wallpaper patterns \cite{conway2016symmetries}, the symmetry type of the hexagons' tiling must be $2\star22$, as in Figure~\ref{fig-421}. 
    Therefore, by Proposition~\ref{prop-upstairsDownstairsSymm},
    the symmetry type of the (3,6)-fullerene is $2\star 2$
\end{proof}

\begin{lemma}
    Suppose a (3,6)-fullerene has signature $(s, b, \dfrac{-(b+1)}{2})$ where $b$ is odd. Then the (3,6)-fullerene has symmetry type $\star 222$. 
    \label{lem-star222}
\end{lemma}

\begin{proof}
    The offset of $\dfrac{-(b+1)}{2} $ means that if a special hexagon, which we will call hexagon $A$, is shifted in the approximately SW to NE direction through $b+1$ columns of hexagons, it ends up $\dfrac{-(b+1)}{2} \pmod{s+1}$ hexagons below a special hexagon. Call this second special hexagon $B$. Note that every column that hexagon $A$ is  shifted over in the SW to NE direction puts it a height of half a hexagon above its original horizontal position. Therefore, if hexagon $A$ ends up $\dfrac{-(b+1)}{2} \pmod{s+1}$ hexagons below a hexagon $B$ after shifting $b+1$ columns, it must have been at a height of $\dfrac{-(b+1)}{2} + \dfrac{b+1}{2} \pmod{s+1} \equiv 0 \pmod{s+1}$ hexagons below hexagon $B$ to begin with. Therefore, the hexagon $A$ lies precisely at the same height as a special hexagon in the column of hexagons containing $B$.

    So all special hexagons are laid out in a grid at equal heights as the special hexagons in neighboring columns of special hexagons. There are horizontal mirror lines going through the rows of special hexagons and halfway in between the rows of hexagons. There are vertical mirror lines going through columns of special hexagons and halfway between columns of special hexagons. There are exactly four distinct points at the intersection of pairs of mirrors that are not equivalent to each other by an isometry of the plane that preserves special hexagons. See Figure~\ref{fig-433}. So by the classification of wallpaper patterns \cite{conway2016symmetries}, the symmetry type of the hexagonal tiling is $\star 2222$. Therefore, by Proposition~\ref{prop-upstairsDownstairsSymm} the (3,6)-fullerene has symmetry type $\star 222$.
\end{proof}

\begin{figure}[ht]
\centering
\includegraphics[height = 8 cm]{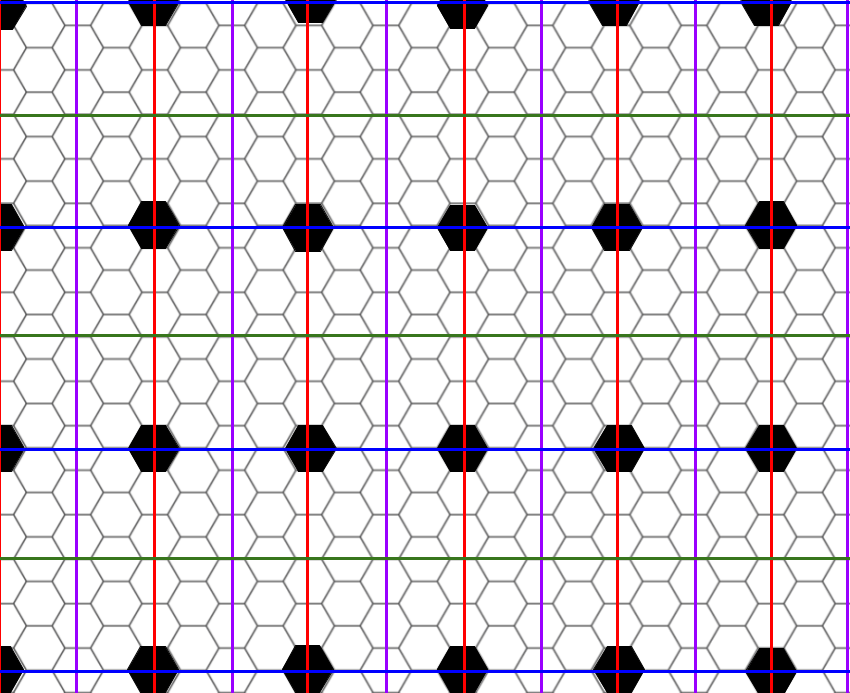}
    \caption {The $(3,6)$-fullerene with signature (4, 3, 3) has $\star 2222$ symmetry upstairs and $\star 222$ symmetry downstairs. Mirror lines are drawn.} 
    \label{fig-433}
    \end{figure}

As in \cite{green2025enumerate}, we say that a signature $(s, b, f )$ is \emph{self-mirror} if the signature $(s, b, s-b-f \pmod{s+1})$ is the exact same signature as $(s, b, f \pmod{s+1})$, not just equivalent. Self-mirror signatures are closely related to mirror symmetry: Corollary~6.3 of \cite{green2025enumerate} says that every $(3,6)$-fullerene with mirror symmetry has a unique self-mirror signature and each self-mirror signature represents a $(3,6)$-fullerene with mirror symmetry.

\begin{lemma}
    A signature $(s, b, f)$ is self-mirror if and only if either: 
    \begin{enumerate}
        \item $s$ and $b$ are both even and $f \equiv \dfrac{s-b}{2} \pmod{s+1}$
        \item $s$ and $b$ are both odd and $f \equiv \dfrac{s-b}{2} \pmod{s+1}$
        \item $s$ and $b$ are both odd and $f \equiv -\dfrac{b+1}{2} \pmod{s+1}$
        \item $s$ is even and $b$ is odd and $f \equiv -\dfrac{b+1}{2} \pmod{s+1}$
    \end{enumerate}
    %$s$ and $b$ have the same parity and $f \equiv -\dfrac{s-b}{2} \pmod{s+1}$ or $b$ is odd and $f = -\dfrac{b+1}{2}$. 
    \label{lem-selfmirroralgebra}
\end{lemma}

\begin{proof}
Consider a (3, 6)-fullerene with signature $(s, b, f)$. If Case 1 or Case 2 hold, i.e. $s$ and $b$ have the same parity and $f \equiv \dfrac{s-b}{2} \pmod{s+1}$, then $2f \equiv s- b \pmod{s+1}$, so $f \equiv s - b - f \pmod{s+1}$. Therefore, the signature $(s, b, f)$ is self-mirror.
If Case 3 or Case 4 hold, i.e. $b$ is odd and $f \equiv -\dfrac{b+1}{2}$, then $2f \equiv -b - 1 \equiv s - b \pmod{s+1}$,  so $f \equiv s- b- f\pmod{s+1}$, so $(s, b, f)$ is self-mirror.

Conversely, suppose that the signature $(s, b, f)$ is self-mirror. Then $f \equiv s- b- f\pmod{s+1}$, so $2f \equiv s - b \pmod{s+1}$. If $s-b$ is odd and $s+1$ is even, then there are no values of $f$ that satisfy this equation  Therefore, either $s$ and $b$ have the same parity, or else $s$ is even. We can divide these possibilities into three cases: A) $s$ and $b$ are both odd, B) $s$ and $b$ are both even, and C) $s$ is even and $b$ is odd (since $s$ even and $b$ even is already covered by case B).

Consider the case where $s$ and $b$ are both even. Since $\gcd(2, s+1) = 1$ $2f \equiv s- b \pmod{s+1}$ has a unique solution for $f$ $\pmod{s+1}$.  So $f$ must be $\dfrac{s-b}{2} \pmod{s+1}$, and conclusion 1 of the lemma holds. 

Next, consider the case where $s$ and $b$  are both are odd. Since $\gcd(2, s+1) = 2$ and $2 \mid (s-b)$, the equation $2f \equiv s- b \pmod{s+1}$ has two solutions for $f$, which must be $\dfrac{s-b}{2} \pmod{s+1}$ and $-\dfrac{b+1}{2} \pmod{s+1}$. Note that these two solutions are distinct, since if $\dfrac{s - b}{2}\equiv -\dfrac{b+1}{2} \pmod{s+1}$ then $\dfrac{s+1}{2} \equiv 0 \pmod{s+1}$, which is impossible. Conclusions 2 and 3 of the lemma hold in this case.

Finally, consider the case where $s$ is even and $b$ is odd. Since $\gcd(2, s+1) = 1$, the equation $2f \equiv s- b \pmod{s+1}$ has one solution for $f$, which must be $f \equiv -\dfrac{b+1}{2} \pmod{s+1}$, and conclusion 4 of the lemma holds.

\end{proof}

\begin{proposition}
A (3,6)-fullerene has mirror symmetry if and only if it has one of the following spherical symmetry types in orbifold notation: $\star 332$, $2\star 2$, or $\star 222$. These are the symmetry types $T_d$, $D_{2d}$, and $D_{2h}$, respectively, in Sch\"onflies notation. 
\label{prop-mirrorSymmetryClassify}

Furthermore,
\begin{enumerate}
\item[(a)] A (3,6)-fullerene has $\star 332$ symmetry if and only if the (3,6)-fullerene has coinciding signatures and for its signature $(s, b, f)$, $s$ and $b$ are either both even or both odd, and $f \equiv \dfrac{s-b}{2} \pmod{s+1}$. 
\item[(b)] A (3,6)-fullerene has $2\star 2$ symmetry if and only if it does not have coinciding signatures and it has a signature
$(s, b, f)$ such that $s$ and $b$ are either both even or both odd, and $f \equiv \dfrac{s-b}{2} \pmod{s+1}$. 
\item[(c)] A (3,6)-fullerene has $\star 222$ symmetry if and only if it has a signature $(s, b, f)$ such that $b$ is odd and and $f \equiv \dfrac{-(b+1)}{2} \pmod{s+1}$.

\end{enumerate}

It is not possible to have a (3,6)-fullerene with mirror symmetry whose signatures all have $s$ odd and $b$ even.

\end{proposition}

\begin{proof}
Suppose that a (3, 6)-fullerene has mirror symmetry. By Corollary~6.3 of \cite{green2025enumerate}, if has a self-mirror signature $(s, b, f)$. Lemma~\ref{lem-selfmirroralgebra} gives four options. 

Case 1: $s$ is even and $b$ is even and $f \equiv \dfrac{s-b}{2} \pmod{s+1}$.  Then by Lemmas~\ref{lem-star332} and \ref{lem-2star2}, the (3,6)-fullerene has symmetry type $\star 332$ if its signatures coincide and $2\star2$ otherwise.

Case 2: $s$ is odd and $b$ is odd and $f \equiv \dfrac{s-b}{2} \pmod{s+1}$.
By Lemmas~\ref{lem-star332} and \ref{lem-2star2}, the (3,6)-fullerene has symmetry type $\star 332$ if its signatures coincide and $2\star2$ otherwise. 

Case 3: $s$ is odd and $b$ is odd and $f \equiv -\dfrac{b+1}{2} \pmod{s+1}$. By Lemma~\ref{lem-star222}, the (3,6)-fullerene has symmetry type $\star 222$.

Case 4: $s$ is even and $b$ is odd and $f \equiv -\dfrac{b+1}{2} \pmod{s+1}$. By Lemma~\ref{lem-star222}, the (3,6)-fullerene has symmetry type $\star 222$.

This analysis shows that any (3,6)-fullerene with mirror symmetry must have symmetry type $\star 332$ or $2\star2$ or  $\star 222$. 

Conversely, any (3,6)-fullerene with symmetry type $\star 332$ or $2\star2$  or $\star 222$ must have mirror symmetry, since mirror symmetry is a feature of each of these symmetry types. This proves the first statement of the proposition. 

Now consider statements (a), (b), and (c). The ``if'' directions of these statements follow directly from Lemmas~\ref{lem-star332}, \ref{lem-2star2}, and \ref{lem-star222}. For the ``only if'' directions, note that any (3, 6)-fullerene with $\star 332$, $2\star2$, or $\star 222$ symmetry must have mirror symmetry, and therefore by Corollary~6.3 of \cite{green2025enumerate}, it must have a signature that is self-mirror. So it must fall into one of the four cases of  Lemma~\ref{lem-selfmirroralgebra}. If it has $\star 222$ symmetry,  one of the last two cases of Lemma~\ref{lem-selfmirroralgebra} must apply, since otherwise, the ``if'' direction of the current lemma that was already established would force a different type of symmetry. This proves the ``only if" direction of part (c). Similarly, if the $(3,6)$-fullerene has $2\star 2$ symmetry, one of the first two cases of Lemma~\ref{lem-selfmirroralgebra} must apply, since otherwise the ``if'' direction of the current lemma would force a different type of symmetry.  Since $\star 332$ symmetry implies 3-fold rotational symmetry and therefore coinciding signatures by Proposition~3.2 of \cite{green2025enumerate}, and $2 \star 2$ implies no 3-fold rotational symmetry and therefore no coinciding signatures, the ``only if'' directions of parts (a) and (b) follow. 

The last statement of the proposition follows from the first statement and the fact that none of cases (a), (b), and (c) allow $s$ to be odd and $b$ to be even in all three signatures.

\end{proof}

Following \cite{deza2005zigzag}, we call a $(3, 6)$ fullerene \emph{tight} if it has no belts. This means that the second coordinate $b$ is zero in each of its three alternative signatures. 

\begin{corollary}
    If $V$ is a multiple of 8, then there are no tight (3,6)-fullerenes with mirror symmetry with $V$ vertices. 
\end{corollary}

\begin{proof}
    Suppose that there is a tight $(3, 6)$-fullerene with mirror symmetry and $V$ vertices, where $V$ is a multiple of 8. Since it has no belts, every signature for this $(3,6)$-fullerene must have second coordinate $b = 0$. Since $V$ is a multiple of 8, $\frac{V}{4}$ is even. Since $\frac{V}{4} = (s+1)(b+1)$, by Equation~\ref{eqn-vCount}, and $b = 0$ in every signature, it follows that $s$ is odd and $b$ is even in every signature. This contradicts Proposition~\ref{prop-mirrorSymmetryClassify}.
\end{proof}

\begin{proposition}
    A (3,6)-fullerene has $332$ symmetry if and only if it has 3-fold rotational symmetry but no mirror symmetry. 
    \label{prop-332}
\end{proposition}

\begin{proof}
    Suppose a (3,6)-fullerene has $332$ symmetry. Then it clearly has 3-fold rotational symmetry and no mirror symmetry. 
    
    Conversely, suppose a (3,6)-fullerene has 3-fold rotational symmetry and no mirror symmetry.  Since the (3,6)-fullerene has no mirror symmetry, by Proposition~6.1 of \cite{green2025enumerate}, the hexagonal tiling that covers it has no mirror symmetry. Since the (3,6)-fullerene does have 3-fold rotational symmetry, by Proposition 3.2 of  \cite{green2025enumerate}, in the hexagonal tiling that covers it, any rotation by $60^\circ$ around the center of a special hexagon takes special hexagons to special hexagons. Therefore, there is a 6-fold rotation in the hexagonal tiling, but no mirror symmetry. So by the classification of wallpaper patterns \cite{conway2016symmetries}, the hexagonal tiling must have symmetry type $632$. By Proposition~\ref{prop-upstairsDownstairsSymm}, the (3,6)-fullerene has symmetry type $332$. See Figure~\ref{fig-602} for an example.
\end{proof}

\begin{figure} 

\centering 

\includegraphics[height = 8 cm]{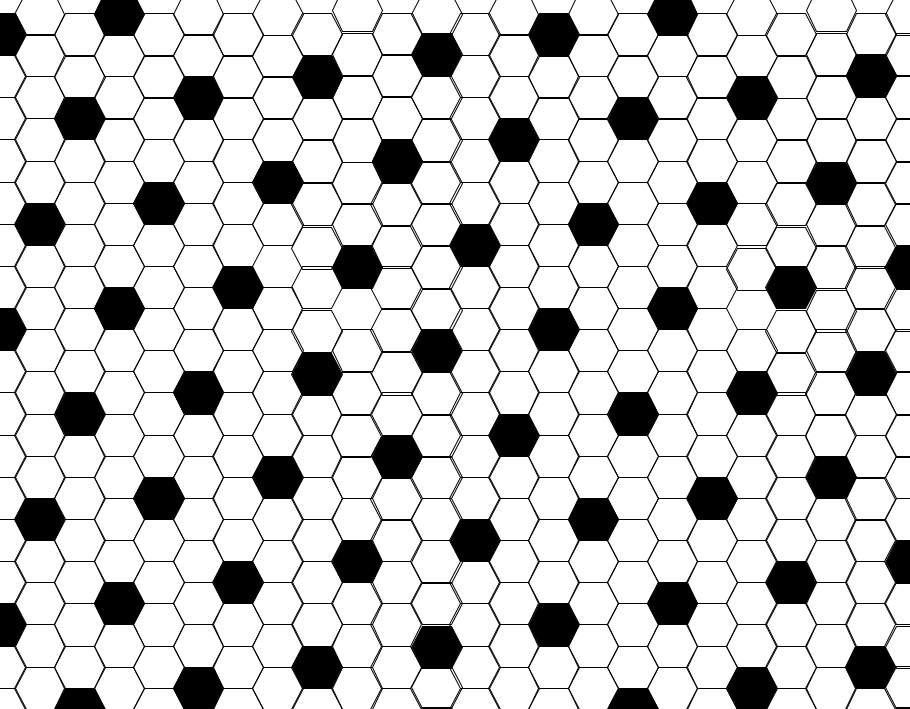}
   \caption {The $(3,6)$-fullerene with signature (6, 0, 2) has 632 symmetry upstairs and 332 symmetry downstairs.}
   \label{fig-602}
   \end{figure}

\begin{lemma}
If the hexagonal cover of a (3,6)-fullerene has glide reflection symmetry, then it also  has mirror lines.
\label{lem-noglides}
\end{lemma}

\begin{proof} Suppose there is a glide line for the hexagonal cover of a (3, 6)-fullerene.
A glide reflection repeated twice is a translation in the direction of the glide line. Since the only directions of translation symmetry for a hexagonal tiling are vertical or at a 30, 60, 90, 120, or 150 degree angle from vertical, the glide line must be in one of these directions.  By rotating the hexagonal tiling if necessary by a multiple of $60^\circ$, we can assume that the glide line is either vertical or horizontal. 

Suppose the glide line is vertical. Since the glide reflection takes special hexagons to special hexagons, it must either cut through a vertical spine or lie halfway between two consecutive spine columns. See Figure~\ref{fig-glideLinePossibleOrImpossible}. Consider the signature $(s, b, f)$ for the vertical orientation (vertical spines). By the definition of offset, moving a special hexagon in the approximately SW to NE direction by $b+1$ steps would land it $f$ hexagons below a special hexagon. But moving it, instead, in the SE to NW direction by $b+1$ steps would also land it $f$ hexagons below a special hexagon due to the glide symmetry. It follows that for every special hexagon, the two spine columns to the left and to the right contain special hexagons  at the same height. Therefore, there is a vertical mirror line of symmetry through each vertical spine in the hexagonal cover, proving the lemma in the case that a glide line is vertical. 

 \begin{figure} 

 \centering
 
\includegraphics[height = 8 cm]{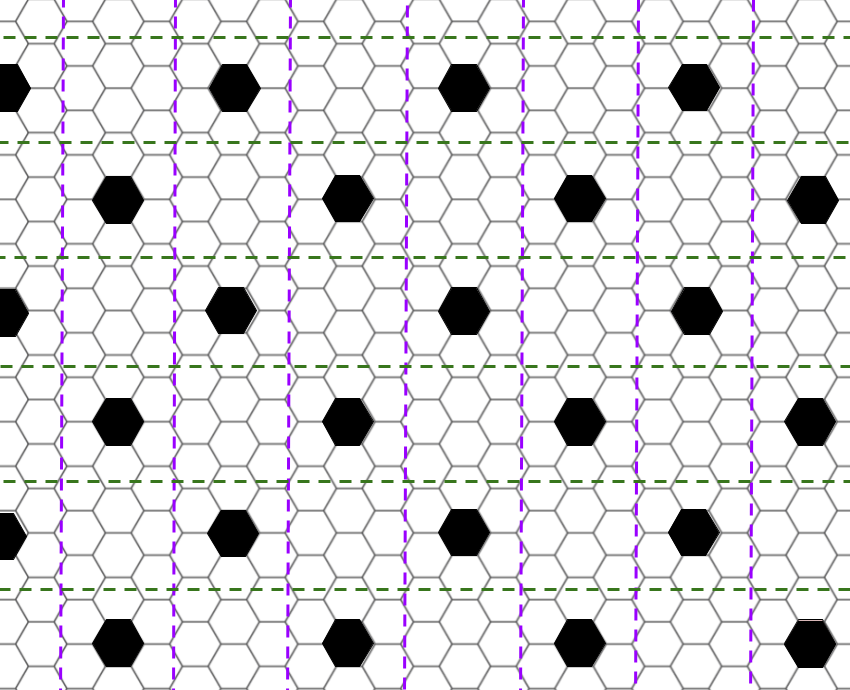}
   \caption {Glide reflection symmetry forces mirror symmetry}
   \label{fig-glideLinePossibleOrImpossible}
   \end{figure}

 If, instead, the glide line is horizontal, then consider a special hexagon that is below this glide line and as close to the glide line as possible. By definition of offset, moving this hexagon in the approximately SW to NE direction by $b+1$ columns puts it $f$ hexagons below a second special hexagon. Equivalently, moving the second special hexagon $b+1$ columns in the NE to SW direction puts it $f$ hexagons above a special hexagon. But due to glide symmetry, moving this second special hexagon $b+1$ columns in the NW to SE direction also puts it $f$ hexagons above a special hexagon. As before, it follows that for every special hexagon, the two spine columns to the left and to the right contain special hexagons  at the same height. Therefore, there is a vertical mirror line of symmetry in the hexagonal cover, proving the lemma in the case that a glide line is horizontal, and finishing the proof.
\end{proof}

The following proposition shows that the only other symmetry type for a $(3, 6)$-fullerene is $222$ symmetry. See Figure~\ref{fig-502} for an example of the corresponding $2222$ symmetry upstairs. 

\begin{figure}

\centering 

\includegraphics[height = 8 cm]{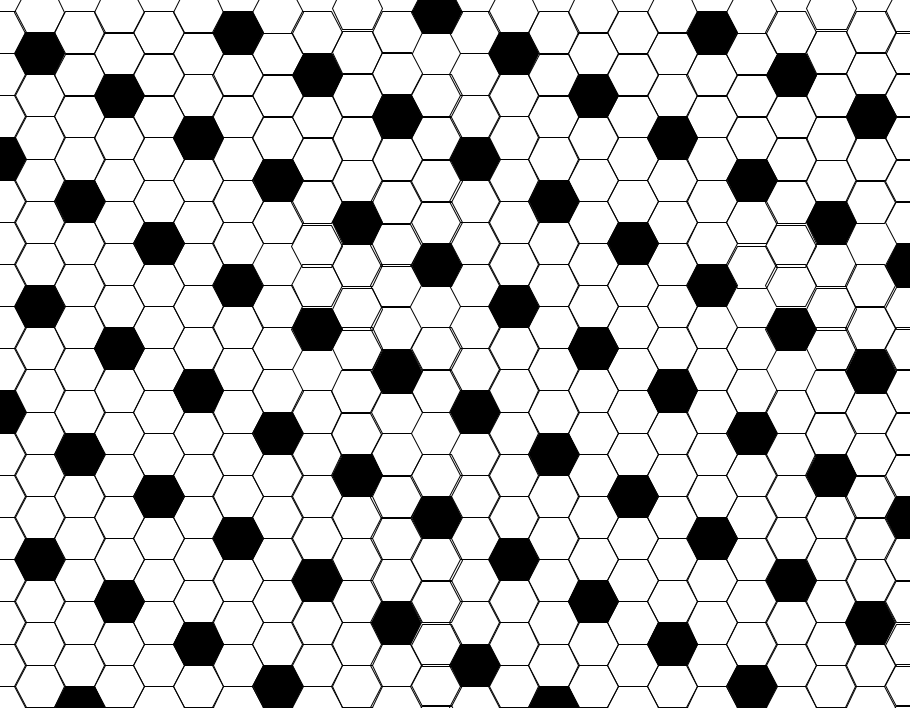}
\caption{The $(3,6)$-fullerene with signature (5, 0, 2) has 2222 symmetry upstairs and 222 symmetry downstairs.}
\label{fig-502}
\end{figure}

\begin{proposition}
    A (3,6)-fullerene has $222$ symmetry if and only if it has no 3-fold rotational symmetry and no mirror symmetry. 
    \label{prop-222}
\end{proposition}

\begin{proof}
    If a (3,6)-fullerene has $222$ symmetry, then it clearly has no 3-fold rotational symmetry and no mirror symmetry. 
    
    Conversely, suppose a (3,6)-fullerene has no 3-fold rotational symmetry and no mirror symmetry.  Since the (3,6)-fullerene has no 3-fold rotational symmetry, by Proposition~3.2 of \cite{green2025enumerate}, in the hexagonal tiling that covers it, there is no $120^\circ$ rotation that takes special hexagons to special hexagons, so there is no 3-fold rotational symmetry of the hexagonal tiling, and therefore no 6-fold rotational symmetry. The only rotational symmetry that the hexagonal covering can have is 2-fold, since only 2-fold, 3-fold, and 6-fold rotations can take a hexagon to itself. In fact, the hexagonal tiling does have 2-fold rotations, including rotations around the center of each special hexagon and rotations around a point in the middle of each spine. However, there are no mirror lines in the hexagonal tiling cover, by Proposition 6.1 of \cite{green2025enumerate}, since there is no reflection symmetry the (3,6)-fullerene itself.  There can be no glide reflection symmetry in the hexagonal tiling cover, by Lemma~\ref{lem-noglides}. 

    So a (3,6)-fullerene with no mirror symmetry and no 3-fold rotational symmetry must have a hexagonal cover with 2-fold rotational symmetry and no mirror symmetry and no glide reflection symmetry. In addition, the hexagonal covering cannot have 4-fold rotational symmetry, since 4-fold rotational symmetry would not preserve hexagons. By the classification of wallpaper patterns \cite{conway2016symmetries}, the hexagonal covering must have symmetry type $2222$. Therefore, by Proposition~\ref{prop-upstairsDownstairsSymm}, the (3,6)-fullerene must have symmetry type $222$.
\end{proof}

\begin{proposition}
    Every (3,6)-fullerene has one of the following five symmetry types: $\star 332$, $332$, $2\star 2$,
     $\star 222$,   or $222$.
    \label{prop-listSymmTypes}
\end{proposition}

\begin{proof}
This follows from Propositions~\ref{prop-mirrorSymmetryClassify}, \ref{prop-332}, and \ref{prop-222} .
\end{proof}

\section{Counts of (3,6)-fullerenes by symmetry type}
\label{sec-countsBySymmType}

This section enumerates the $(3,6)$ fullerenes in each of the five possible symmetry types: $\star 332$, $332$, $2\star 2$, $\star 222$, and $222$. All counts in this section, including Table~\ref{table-countBySymmType}, include godseyes.  (However, the existence results of Section~\ref{sec-conj1And2} exclude them, to match reference \cite{deza2005zigzag}.)

\begin{proposition}
Suppose $\frac{V}{4}$ factors as $\displaystyle 2^x 3^y \prod_{i = 1}^z p_i^{k_i} \prod_{j = 1}^w q_j^{\ell_j}$ where $x, y, z, w \geq 0$, the $p_i$ are primes with $p_i \equiv 1 \pmod 3$ the $q_j$ are primes greater than 2 with $q_j \equiv 2 \pmod 3$, and  $k_i, \ell_j \geq 1$ for $1 \leq i \leq z$ and $1 \leq j \leq w$. 

\begin{enumerate}
\item[(a)] The number of (3,6)-fullerenes with $V$ vertices with $\star 332$ symmetry is 1 if all $k_i$ are even, all $\ell_j$ are even, and $x$ is even, i.e. if $\frac{V}{4}$ is a perfect square or 3 times a perfect square. The number is 0 otherwise.

\item[(b)] The number of $(3,6)$-fullerenes with $V$ vertices with either $\star 332$ symmetry or $2\star 2$ symmetry is \\$\displaystyle  (y+1)\prod_{i = 1}^z (k_i + 1)\prod_{j = 1}^w (\ell_j + 1)$ if $x = 0$,  and \\
$\displaystyle (x - 1)(y+1)\prod_{i = 1}^z (k_i + 1)\prod_{j = 1}^w (\ell_j + 1)$ if $x \geq 1$.

Therefore, the number of $(3,6)$-fullerenes with $V$ vertices and $2 \star 2$ symmetry can be found by subtracting the number in part (a) from this number.

\item[(c)] The number of (3,6)-fullerenes with $V$ vertices with $\star 222$ symmetry is \\
$\displaystyle x(y+1) \prod_{i = 1}^z (k_i + 1)\prod_{j = 1}^w (\ell_j + 1)$.

\item[(d)] The number of (3,6)-fullerenes with $V$ vertices and mirror symmetry is the sum of the numbers in (b) and (c), which is \\
$\displaystyle  (y+1)\prod_{i = 1}^z (k_i + 1)\prod_{j = 1}^w (\ell_i + 1)$ if $x = 0$,   and \\
$\displaystyle  (2x - 1 )(y + 1)\prod_{i = 1}^z (k_i + 1)\prod_{j = 1}^w (\ell_j + 1)$ if $x \geq 1$.

\item[(e)] The number of (3,6)-fullerenes with $V$ vertices and $332$ symmetry is \\
$\displaystyle \prod_{i = 1}^z (k_i + 1) - 1$ if all $k_i$ are even, all $\ell_j$ are even, and $x$ is even, \\
$\displaystyle \prod_{i = 1}^z (k_i + 1)$ if all $\ell_j$ are even and $x$ is even, but at least one $k_i$ is odd, and \\
0 if $x$ is odd or any $\ell_j$ is odd.

\item[(f)] The total number of $(3,6)$-fullerenes with $V$ vertices is given by 
\[ \displaystyle  \frac{1}{3} \left[ \left( \dfrac{2^{x+1} - 1}{2 - 1} \right) \left( \dfrac{3^{y+1} - 1}{3 - 1} \right) \prod_{i = 1}^z \dfrac{p_i^{k_i +1} - 1}{p_i - 1}  \prod_{j = 1}^w \dfrac{q_j^{\ell_j + 1} - 1}{q_j - 1}  \right] + \frac{2}{3} \prod_{k = 1}^z (k_i + 1) \] 
if all $\ell_i$ are even and $x$ is even, and \[ \displaystyle  \frac{1}{3} \left[ \left( \dfrac{2^{x+1} - 1}{2 - 1} \right) \left( \dfrac{3^{y+1} - 1}{3 - 1} \right) \prod_{i = 1}^z \dfrac{p_i^{k_i +1} - 1}{p_i - 1}  \prod_{j = 1}^w \dfrac{q_j^{\ell_j + 1} - 1}{q_j - 1}  \right] \] otherwise.

Therefore, the number of $(3,6)$-fullerenes with $V$ vertices and $222$ symmetry can be found by subtracting the number in part (d) and the number in part (e) from this number.
\end{enumerate}

\label{prop-count}
\end{proposition}

\begin{proof}
(a) This is a restatement of Proposition~7.2 of \cite{green2025enumerate}.

(b) By Proposition~\ref{prop-mirrorSymmetryClassify}, a (3,6)-fullerene has $\star 332$ symmetry or  $2\star 2$ if and only if it has a signature $(s, b, f)$ where $s$ and $b$ are either both even or both odd and $f \equiv \dfrac{s-b}{2} \pmod {s+1}$. This signature is self-mirror by Lemma~\ref{lem-selfmirroralgebra}. If the (3,6)-fullerene does not have 3-fold rotational symmetry, then by Lemma 6.2 of \cite{green2025enumerate}, its other two signatures are not self-mirror and therefore can't be of this form. Therefore, there is a one-to-one correspondence between (3,6)-fullerenes with either $\star 332$ symmetry  or $2\star 2$ and signatures of the form $\left(s, b, \dfrac{s-b}{2}\right)$ where $s$ and $b$ have the same parity. Since $\frac{V}{4} = (s+1)(b+1)$ by Equation~\ref{eqn-vCount}, the number of (3,6)-fullerenes with $V$ vertices and $\star 332$ or $2\star 2$  symmetry is equal to the number of ways of factoring $\displaystyle \frac{V}{4} = 2^x 3^y \prod_{i = 1}^z p_i^{k_i}$ into two factors of the same parity. There are $\displaystyle (y+1) \prod_{i = 1}^z (k_i + 1)\prod_{j = 1}^w (\ell_j + 1)$ ways if $x = 0$. If $x \geq 1$, then there are $\displaystyle (x-1) (y+1)\prod_{i = 1}^z (k_i + 1)\prod_{j = 1}^w (\ell_j + 1)$ ways, where the $x-1$ comes from the fact that both of the factors will need to be even.

(c) By Proposition~\ref{prop-mirrorSymmetryClassify}, a (3,6)-fullerene has $\star 222$ symmetry if and only if it has a signature of the form $(s, b, f)$ where $b$ is odd and $f \equiv \dfrac{-(b+1)}{2} \pmod {s+1}$. This signature is self-mirror by Lemma~\ref{lem-selfmirroralgebra}. Since the (3, 6)-fullerene does not have 3-fold rotational symmetry, its other two other signatures are not self-mirror by Lemma 6.2 of \cite{green2025enumerate}. Therefore, the two other signatures cannot be of the form $(s, b, f)$ where $b$ is odd and $f \equiv \dfrac{-(b+1)}{2} \pmod {s+1}$. Therefore, there is a one-to-one correspondence between (3,6)-fullerenes with $\star 222$ symmetry and signatures $(s, b, \dfrac{-(b+1)}{2})$ where $b$ is odd. Let $V$ represent the number of vertices of a (3,6)-fullerene. By Equation~\ref{eqn-vCount},  $\frac{V}{4} = (s+1)(b+1)$, so the number of (3,6)-fullerenes with $V$ vertices and $\star 222$ symmetry is equal to the number of pairs $(s, b)$ where $b$ is odd and $\frac{V}{4} = (s+1)(b+1)$. This is equal to the number of even factors $b+1$ of $\frac{V}{4}$. Even factors have to have at least one factor of 2, can have between 0 and y copies of $3$, and can have between 0 and $k_i$ factors of $p_i$ for each $i$, $1 \leq i \leq z$ and between $0$ and $\ell_j$ factors of $q_j$ for $1 \leq j \leq w$. There are $\displaystyle x (y+1) \prod_{i = 1}^z (k_i + 1)\prod_{j = 1}^w (\ell_i + 1)$ such factors.

(d) This follows immediately from the fact that the (3,6)-fullerenes with mirror symmetry are exactly the (3,6)-fullerenes with $\star 332$ or $2\star 2$ or $\star 222$ symmetry (Proposition~\ref{prop-mirrorSymmetryClassify}). It was also proved in Proposition~7.3 of \cite{green2025enumerate}.

(e) By Theorem 5.1 of \cite{green2025enumerate}, the number of $(3,6)$-fullerenes with 3-fold rotational symmetry is $\displaystyle \prod_{i = 1}^z (k_i + 1)$ if all $\ell_j$ are even and $x$ is even, and $0$ otherwise. The statement follows from this and part (a). 

(f) The formula for the total numbrer of $(3,6)$-fullerenes with $V$ vertices is given in Theorem~5.2 of \cite{green2025enumerate}. All $(3,6)$-fullerenes have one of the following five symmetry types by Theorem~\ref{prop-listSymmTypes}: $\star 332$, $332$, $2\star 2$, $\star 222$, $222$.  Therefore, the  number of $(3,6)$-fullerenes with $V$ vertices and $222$ symmetry is equal to the total number of $(3,6)$-fullerenes with $V$ vertices minus the number with mirror symmetry minus the number with $332$ symmetry.

\end{proof}

Table \ref{table-countBySymmType} gives the number $(3,6)$-fullerenes by symmetry type for up to 400 vertices. The 2-connected graphs called godeyes are included in these counts. Note that this enumeration counts left-handed and right-handed fullerenes as distinct. If, instead, we want the counts in which left-handed and right-handed versions are considered the same, then the numbers in the $332$ and $222$ columns need to be halved and the number in the total column needs to be reduced accordingly.

\begin{table}
\refstepcounter{table}\label{table-countBySymmType}
\begin{tabular}{|c|c|c|c|c|c|c||c|c|c|c|c|c|c|} \hline
$V/4$	&	$\star 332$ & $332$ & 	$2\star 2$	&	$\star 222$		&	$222$	&	Total	&	V/4	&	$\star 332$	& $332$	& 	$2\star 2$	& $\star 222$	&			$222$	&	Total	\\ \hline \hline
1	&	1	&	0	&	0	&	0	&	0	&	1	&	51	&	0	&	0	&	4	&	0	&	20	&	24	\\ \hline
2	&	0	&	0	&	0	&	1	&	0	&	1	&	52	&	0	&	2	&	2	&	4	&	26	&	34	\\ \hline
3	&	1	&	0	&	1	&	0	&	0	&	2	&	53	&	0	&	0	&	2	&	0	&	16	&	18	\\ \hline
4	&	1	&	0	&	0	&	2	&	0	&	3	&	54	&	0	&	0	&	0	&	4	&	36	&	40	\\ \hline
5	&	0	&	0	&	2	&	0	&	0	&	2	&	55	&	0	&	0	&	4	&	0	&	20	&	24	\\ \hline
6	&	0	&	0	&	0	&	2	&	2	&	4	&	56	&	0	&	0	&	4	&	6	&	30	&	40	\\ \hline
7	&	0	&	2	&	2	&	0	&	0	&	4	&	57	&	0	&	2	&	4	&	0	&	22	&	28	\\ \hline
8	&	0	&	0	&	2	&	3	&	0	&	5	&	58	&	0	&	0	&	0	&	2	&	28	&	30	\\ \hline
9	&	1	&	0	&	2	&	0	&	2	&	5	&	59	&	0	&	0	&	2	&	0	&	18	&	20	\\ \hline
10	&	0	&	0	&	0	&	2	&	4	&	6	&	60	&	0	&	0	&	4	&	8	&	44	&	56	\\ \hline
11	&	0	&	0	&	2	&	0	&	2	&	4	&	61	&	0	&	2	&	2	&	0	&	18	&	22	\\ \hline
12	&	1	&	0	&	1	&	4	&	4	&	10	&	62	&	0	&	0	&	0	&	2	&	30	&	32	\\ \hline
13	&	0	&	2	&	2	&	0	&	2	&	6	&	63	&	0	&	2	&	6	&	0	&	28	&	36	\\ \hline
14	&	0	&	0	&	0	&	2	&	6	&	8	&	64	&	1	&	0	&	4	&	6	&	32	&	43	\\ \hline
15	&	0	&	0	&	4	&	0	&	4	&	8	&	65	&	0	&	0	&	4	&	0	&	24	&	28	\\ \hline
16	&	1	&	0	&	2	&	4	&	4	&	11	&	66	&	0	&	0	&	0	&	4	&	44	&	48	\\ \hline
17	&	0	&	0	&	2	&	0	&	4	&	6	&	67	&	0	&	2	&	2	&	0	&	20	&	24	\\ \hline
18	&	0	&	0	&	0	&	3	&	10	&	13	&	68	&	0	&	0	&	2	&	4	&	36	&	42	\\ \hline
19	&	0	&	2	&	2	&	0	&	4	&	8	&	69	&	0	&	0	&	4	&	0	&	28	&	32	\\ \hline
20	&	0	&	0	&	2	&	4	&	8	&	14	&	70	&	0	&	0	&	0	&	4	&	44	&	48	\\ \hline
21	&	0	&	2	&	4	&	0	&	6	&	12	&	71	&	0	&	0	&	2	&	0	&	22	&	24	\\ \hline
22	&	0	&	0	&	0	&	2	&	10	&	12	&	72	&	0	&	0	&	6	&	9	&	50	&	65	\\ \hline
23	&	0	&	0	&	2	&	0	&	6	&	8	&	73	&	0	&	2	&	2	&	0	&	22	&	26	\\ \hline
24	&	0	&	0	&	4	&	6	&	10	&	20	&	74	&	0	&	0	&	0	&	2	&	36	&	38	\\ \hline
25	&	1	&	0	&	2	&	0	&	8	&	11	&	75	&	1	&	0	&	5	&	0	&	36	&	42	\\ \hline
26	&	0	&	0	&	0	&	2	&	12	&	14	&	76	&	0	&	2	&	2	&	4	&	40	&	48	\\ \hline
27	&	1	&	0	&	3	&	0	&	10	&	14	&	77	&	0	&	0	&	4	&	0	&	28	&	32	\\ \hline
28	&	0	&	2	&	2	&	4	&	12	&	20	&	78	&	0	&	0	&	0	&	4	&	52	&	56	\\ \hline
29	&	0	&	0	&	2	&	0	&	8	&	10	&	79	&	0	&	2	&	2	&	0	&	24	&	28	\\ \hline
30	&	0	&	0	&	0	&	4	&	20	&	24	&	80	&	0	&	0	&	6	&	8	&	48	&	62	\\ \hline
31	&	0	&	2	&	2	&	0	&	8	&	12	&	81	&	1	&	0	&	4	&	0	&	36	&	41	\\ \hline
32	&	0	&	0	&	4	&	5	&	12	&	21	&	82	&	0	&	0	&	0	&	2	&	40	&	42	\\ \hline
33	&	0	&	0	&	4	&	0	&	12	&	16	&	83	&	0	&	0	&	2	&	0	&	26	&	28	\\ \hline
34	&	0	&	0	&	0	&	2	&	16	&	18	&	84	&	0	&	2	&	4	&	8	&	62	&	76	\\ \hline
35	&	0	&	0	&	4	&	0	&	12	&	16	&	85	&	0	&	0	&	4	&	0	&	32	&	36	\\ \hline
36	&	1	&	0	&	2	&	6	&	22	&	31	&	86	&	0	&	0	&	0	&	2	&	42	&	44	\\ \hline
37	&	0	&	2	&	2	&	0	&	10	&	14	&	87	&	0	&	0	&	4	&	0	&	36	&	40	\\ \hline
38	&	0	&	0	&	0	&	2	&	18	&	20	&	88	&	0	&	0	&	4	&	6	&	50	&	60	\\ \hline
39	&	0	&	2	&	4	&	0	&	14	&	20	&	89	&	0	&	0	&	2	&	0	&	28	&	30	\\ \hline
40	&	0	&	0	&	4	&	6	&	20	&	30	&	90	&	0	&	0	&	0	&	6	&	72	&	78	\\ \hline
41	&	0	&	0	&	2	&	0	&	12	&	14	&	91	&	0	&	4	&	4	&	0	&	32	&	40	\\ \hline
42	&	0	&	0	&	0	&	4	&	28	&	32	&	92	&	0	&	0	&	2	&	4	&	50	&	56	\\ \hline
43	&	0	&	2	&	2	&	0	&	12	&	16	&	93	&	0	&	2	&	4	&	0	&	38	&	44	\\ \hline
44	&	0	&	0	&	2	&	4	&	22	&	28	&	94	&	0	&	0	&	0	&	2	&	46	&	48	\\ \hline
45	&	0	&	0	&	6	&	0	&	20	&	26	&	95	&	0	&	0	&	4	&	0	&	36	&	40	\\ \hline
46	&	0	&	0	&	0	&	2	&	22	&	24	&	96	&	0	&	0	&	8	&	10	&	66	&	84	\\ \hline
47	&	0	&	0	&	2	&	0	&	14	&	16	&	97	&	0	&	2	&	2	&	0	&	30	&	34	\\ \hline
48	&	1	&	0	&	5	&	8	&	28	&	42	&	98	&	0	&	0	&	0	&	3	&	54	&	57	\\ \hline
49	&	1	&	2	&	2	&	0	&	16	&	21	&	99	&	0	&	0	&	6	&	0	&	46	&	52	\\ \hline
50	&	0	&	0	&	0	&	3	&	28	&	31	&	100	&	1	&	0	&	2	&	6	&	64	&	73	\\ \hline
\end{tabular}
  \begin{minipage}{1.0 \textwidth}
\small {Table 1: Enumeration of $(3, 6)$-fullerenes by symmetry type for $\leq 400$ vertices. Godseyes are included and left-handed and right-handed versions of chiral fullerenes are counted as distinct.}
 \end{minipage}
\end{table}

\section{Existence of fullerenes with specified symmetry type and number of vertices}
\label{sec-conj1And2}

In this section, we prove parts (i) and (ii) and the first statement part (iii) of Conjecture~5.7 of \cite{deza2005zigzag} regarding the existence of (3,6)-fullerenes with $\star 222$, $2\star 2$, and $222$ symmetry types. We also prove statements about the existence of (3,6)-fullerenes with $\star332$ and $332$ symmetry.

Note that the polyhedra considered in \cite{deza2005zigzag} are (3, 6)-fullerenes that are 3-connected, and therefore not godseyes. Therefore, to match \cite{deza2005zigzag}, all existence results in this section exclude godseyes. This in contrast to Section~\ref{sec-countsBySymmType}, in which all enumerations include godseyes. From \cite{green2024polyhedra}, a (3,6)-fullerene is a godseye if and only if its signatures are of the form: $(n,0,0)$, $(n,0,n)$, and $(0,n,0)$ for some number $n \geq 1$.

Part (i) of Conjecture~5.7 can be restated as:
\begin{proposition}
There exists a (3,6)-fullerene that is not a godseye with $V$ vertices with $\star 222$ symmetry if and only if $\frac{V}{4}$ is even and $V \geq 16$. There are no tight (3,6)-fullerenes with $\star 222$ symmetry. 

%Original version:
%     $3_n(D_{2h})$ exists if and only if $\dfrac{n}{4}$ is even and $n \geq 16$. There are no tight $3_n(D_{2h})$.
\end{proposition}

\begin{proof}
    Assume there is a (3,6)-fullerene that is not a godseye with $V$ vertices that has $\star 222$ symmetry. By Proposition~\ref{prop-mirrorSymmetryClassify}, the (3,6)-fullerene has a signature $(s, b, f)$ such that $b$ is odd and $f \equiv \dfrac{-(b+1)}{2} \pmod {s+1}$. By Equation~\ref{eqn-vCount}, $\frac{V}{4} = (s+1)(b+1)$. Since $b$ is odd, $\frac{V}{4}$ is even, and $b \geq 1$. Since the (3,6)-fullerene is not a godseye, it does not have signature $(0, n, 0)$ for $n \geq 1$. So either it has signature $(0,0,0)$ or else $s \geq 1$. But the $(3,6)$-fullerene with signature $(0,0,0)$ is a tetrahedron with $\star 332$ symmetry, so we have $s \geq 1$. So $\frac{V}{4} = (s+1)(b+1) \geq 4$, i.e., $V \geq 16$. 

    Conversely, suppose that $\frac{V}{4}$ is even and $V \geq 16$. Let $s = \frac{V}{8}-1$ and let $b = 1$. Let $f = s$. Consider the (3,6)-fullerene $(s, 1, s)$. This (3, 6)-fullerene has $V$ vertices, since $4(s+1)(b+1) = 4(\frac{V}{8} - 1 + 1)(1 + 1)= V$.
    Since $V \geq 16$  $s = \frac{V}{8} - 1 \geq 1$. Since $s > 0$ and $b > 0$, this (3,6)-fullerene's signature is not of the form  $(n, 0, 0)$, $(n, 0, n)$ or $(0, n, 0)$ for any $n \geq 1$, so it is not a godseye. Since $f = \dfrac{-(b+1)}{2} \pmod{s+1}$, by Proposition~\ref{prop-mirrorSymmetryClassify}, it has $\star 222$ symmetry. So there exists a (3,6)-fullerene that is not a godseye with $V$ vertices and $\star 222$ symmetry. 

    By Proposition~\ref{prop-mirrorSymmetryClassify}, any (3,6)-fullerene with $\star 222$ symmetry has a signature with $b$ odd. Therefore, it has a signature with $b \neq 0$, so it cannot be tight.

\end{proof}

\begin{lemma}
    For $s_1$ even and $s_1 \geq 4$, the (3,6)-fullerene with signature $(s_1, 0, \dfrac{s_1}{2})$ does not have coinciding signatures.
    \label{lem-unhuh}
\end{lemma}

\begin{proof}
Assume for contradiction that this (3,6)-fullerene does have coinciding signatures. So $(s_2, b_2, f_2) = (s_1, b_1, f_1) = (s_1, 0, \frac{s_1}{2})$. 

From Equation~\ref{eqn-f2}, $f_2 \equiv -p_2(b_1 + 1) - (b_2 + 1) \pmod {s_2 + 1} $, where $p_2$ is the smallest positive integer such that $p_2 \cdot f_1 \equiv b_2 + 1 \pmod {s_1 + 1}$. Substituting $0$ for $b_1$ and $b_2$,  $\dfrac{s_1}{2}$ for $f_2$, and $s_1$ for $s_2$, we have that $\dfrac{s_1}{2} \equiv -p_2 - 1 \pmod {s_1 + 1}$, and $p_2 \cdot \dfrac{s_1}{2} \equiv 1 \pmod {s_1 + 1}$. So $p_2 \equiv -\dfrac{s_1}{2} - 1 \equiv \dfrac{s_1}{2} \pmod {s_1 + 1} $, and $\dfrac{s_1}{2}\cdot \dfrac{s_1}{2} \equiv 1 \pmod {s_1 + 1}$. Setting $y = s_1+1$, we have $y \mid (\frac{y-1}{2})^2 - 1$, so $y \mid  \frac{(y-1)^2 - 4}{4} $ so $y  \mid y^2 - 2y -3$. Therefore, $y \mid 3$ so $y = 3$ or $y = 1$. So  $s_1 = 2$ or $s_1 = 0$, which contradicts the assumption that $s_1$ is an even number $\geq 4$. 
\end{proof}

Part (ii) of Conjecture~5.7 of \cite{deza2005zigzag} can be restated as:

\begin{proposition}

There is a (3,6)-fullerene that is not a godseye with $V$ vertices and $2\star 2$ symmetry if and only if either $\frac{V}{4}$ is odd and $V \geq 20$ or $\frac{V}{8}$ is even and $V \geq 32$. If $\frac{V}{4}$ is odd and $V \geq 20$, then there is a unique tight (3, 6)-fullerene with $2\star 2$ symmetry and $V$ vertices and it has four hexagons which are each adjacent
to two triangles on non-opposite and non-adjacent edges.

%Original version: 
%    $3_n(D_{2d})$ exists if and only if either $\dfrac{n}{4}$ is odd and $n \geq 20$ or $\dfrac{n}{8}$ is even and $n \geq 24$. The graph defined in Theorem 5.1(iii) is a unique $3_n(D_{2d})$ tight graph if $\dfrac{n}{4}$ is odd and $n \geq 20$.
\end{proposition}

\begin{proof}

Suppose that there is a (3,6)-fullerene that is not a godseye with $V$ vertices and $2\star2$ symmetry.  By Proposition~\ref{prop-mirrorSymmetryClassify}, the (3,6)-fullerene does not have coinciding signatures and it does have a signature $(s, b, f)$ such that $s$ and $b$ have the same parity and such that $f \equiv \dfrac{s - b}{2} \pmod {s+1}$. Recall that $\frac{V}{4} = (s+1)(b+1)$ by Equation~\ref{eqn-vCount}.

Case 1: $s$ and $b$ are both even: Then $s+1$ and $b+1$ are both odd, so $\frac{V}{4} = (s+1)(b+1)$  is odd.

We will show that $V \geq 20$ by the following argument. Since the (3,6)-fullerene is not a godseye, it does not have signature $(0, n, 0)$ for any $n \geq 1$. So either it has signature $(0,0,0)$ or else $s \geq 1$. But the $(3,6)$-fullerene with signature $(0,0,0)$ is a tetrahedron with $\star 332$ symmetry so $s \geq 1$. Since $s$ is even, $s \geq 2$. If $b \geq 2$, then $\frac{V}{4} = (s+1)(b+1) \geq 9$ and $V \geq 20$ as claimed. If $b = 0$, then if $s \geq 4$, $\frac{V}{4} \geq 5$ so $V \geq 20$ also. It remains to consider the case when $b = 0$ and $s = 2$. In this case, the signature is $(2, 0, 1)$, which has coinciding signatures, a contradiction. Therefore, $V \geq 20$ for all allowable values of $s$ and $b$. 

Case 2: $s$ and $b$ are both odd: Then $s+1$ and $b+1$ are both even, so $\frac{V}{4} = (s+1)(b+1)$  is divisible by 4, and $\frac{V}{8}$ is even. Since $s$ and $b$ are both odd, $s \geq 1$ and $b \geq 1$. But $s = 1$ and $b = 1$ is impossible, since then the signature would be $(1, 1, 0)$, which is a coinciding signature. So in fact, either $s \geq 3$ or $b \geq 3$, so $\frac{V}{4} = (s+1)(b+1) \geq 8$ and $V \geq 32$ as wanted. 

Conversely, suppose that either $\frac{V}{4}$ is odd and $V \geq 20$ or $\frac{V}{8}$ is even and $V \geq 32$. 

Case 1: Suppose $\frac{V}{4}$ is odd and $V \geq 20$. 

 Consider the (3,6)-fullerene with signature $(s_1, b_1, f_1)$ where $s_1 = \frac{V}{4} - 1$, $b_1 = 0$, and  $f_1 =  \dfrac{s_1}{2}$. Note that this is a (3,6)-fullerene with $V$ vertices, since $4(s_1+1)(b_1+1) = 4(\frac{V}{4} -1 + 1)(0+1) = V$.   It is not a godseye, since its signature is not of the form $(n, 0, 0)$, $(n, 0, n)$ or $(0, n, 0)$ for any $n \geq 1$.

 Since $V \geq 20$ implies that $s_1 = \frac{V}{4} -1 \geq 4$, and $s_1$ is even,  Lemma~\ref{lem-unhuh} shows that the (3,6)-fullerene does not have coinciding signatures. In addition, $s_1$ and $b_1$ are both even and $f = \dfrac{s_1-b_1}{2}$, so by Lemma~\ref{lem-2star2}, the (3,6)-fullerene must have $2 \star 2$ symmetry. We have found a (3,6)-fullerene with $V$ vertices that is not a godseye, with the symmetry type $2 \star 2$.

Case 2: $\frac{V}{8}$ is even and $V \geq 32$. Since $\frac{V}{8}$ is even, $\frac{V}{16}$ is an integer. 

Consider the (3,6)-fullerene with signature $(1, \frac{V}{8} -1, f_1)$, where $f_1 = 1$ if $\frac{V}{16}$ is even and $f_1 = 0$ if $\frac{V}{16}$ is odd. 
Note that this (3,6)-fullerene has $V$ vertices, since $4(s_1 + 1)(b_1 + 1) = 4\cdot 2 \cdot \frac{V}{8} = V$. It is not a godseye, since its signature is not of the form $(n, 0, 0)$, $(n, 0, n)$ or $(0, n, 0)$ for any $n \geq 1$.
In addition, $s_1$ and $b_1$ are both odd. Since $f_1$ and $\frac{V}{16}$ have opposite parity, $f_1  \equiv 1 - \frac{V}{16} \pmod {2}$. Since $\dfrac{s_1 - b_1}{2}  = 1 - \frac{V}{16}$, we have that $f_1  \equiv \dfrac{s_1 - b_1}{2} \pmod {2}$, which means that $f_1 \equiv \dfrac{s_1 - b_1}{2} \pmod {s_1 + 1}$, since $s_1 = 1$. So if this (3,6)-fullerene doesn't have coinciding signatures, by Lemma~\ref{lem-2star2}, it will have $2\star 2$ symmetry.

To verify that it doesn't have coinciding signatures, we will show that $s_2 \neq s_1$. Consider $j_2$, the smallest positive integer such that $j_2 \cdot f_1 \equiv 0 \pmod {s_1 + 1}$, where $s_1 + 1 = 2$. If $\frac{V}{16}$ is even, then $f_1 = 1$ by construction, so $j_2 = 2$, and by Equation~\ref{eqn-s2}, $s_2 = j_2(b_1 + 1) - 1 = 2 (\frac{V}{8}) - 1 = \frac{V}{4} - 1$. Since $V \geq 32$,  $\frac{V}{4} - 1 > 1$, so $s_2 \neq s_1$.
If, instead, $\frac{V}{16}$ is odd, then $f_1 = 0$ by construction, so $j_2$ is the smallest positive integer such that $j_2 \cdot 0 \equiv 0 \pmod 2$. Therefore, $j_2 = 1$ and $s_2 = j_2(b_1 + 1) - 1 = 1 (\frac{V}{8}) - 1 = \frac{V}{8} - 1$. Again, $V \geq 32$ ensures that $s_2 > 1$, so $s_2 \neq s_1$. The (3,6)-fullerene does not have coinciding signatures, so it has symmetry type $2 \star 2$.

Finally, we construct tight (3,6)-fullerenes with $2 \star 2$ symmetry. 
If $\frac{V}{4}$ is odd and $V \geq 20$, then consider the (3,6)-fullerene $(s, b, f)$, where $s = \frac{V}{4} - 1$, $b = 0$, and $f = \frac{s}{2} = \frac{V}{8} - \frac{1}{2}$. We saw in Case 1 above that this (3,6)-fullerene has $2 \star 2$ symmetry. This (3,6)-fullerene has $b = 0$. Since $f = \dfrac{s}{2}$ and $f+ 1 = \dfrac{s + 2}{2}$, the three numbers $f$, $f+1$, and $s+1$ are pairwise relatively prime. Therefore, by Proposition 8 of \cite{green2024polyhedra}, it is tight. Any other tight (3,6)-fullerene with $V$ vertices and $2 \star 2$ symmetry must have $b = 0$ for all signatures, so it must have $\frac{V}{4} = (s+1)(b+1) = s+1$ for all signatures. So by Proposition~\ref{prop-mirrorSymmetryClassify}, it must have a signature of  the form $(\frac{V}{4} - 1, 0, \frac{V}{8} - \frac{1}{2})$. Therefore, there is a unique tight (3,6)-fullerene with $2 \star 2$ symmetry when $\frac{V}{4}$ is odd and $V \geq 20$. See Figure~\ref{fig:2star2Tight} for an example. Because these (3,6)-fullerenes have $2 \star 2$ symmetry with no belts, and with spines of length $\geq 4$, they will each have four hexagons which are each adjacent to two triangles on non-opposite and non-adjacent edges, as described in \cite{deza2005zigzag}, Theorem 5.1 (iii). These hexagons are shaded in light blue in Figure~\ref{fig:2star2Tight}.

\end{proof}

\begin{figure}[ht]
    \centering
    \includegraphics[width=0.7\linewidth]{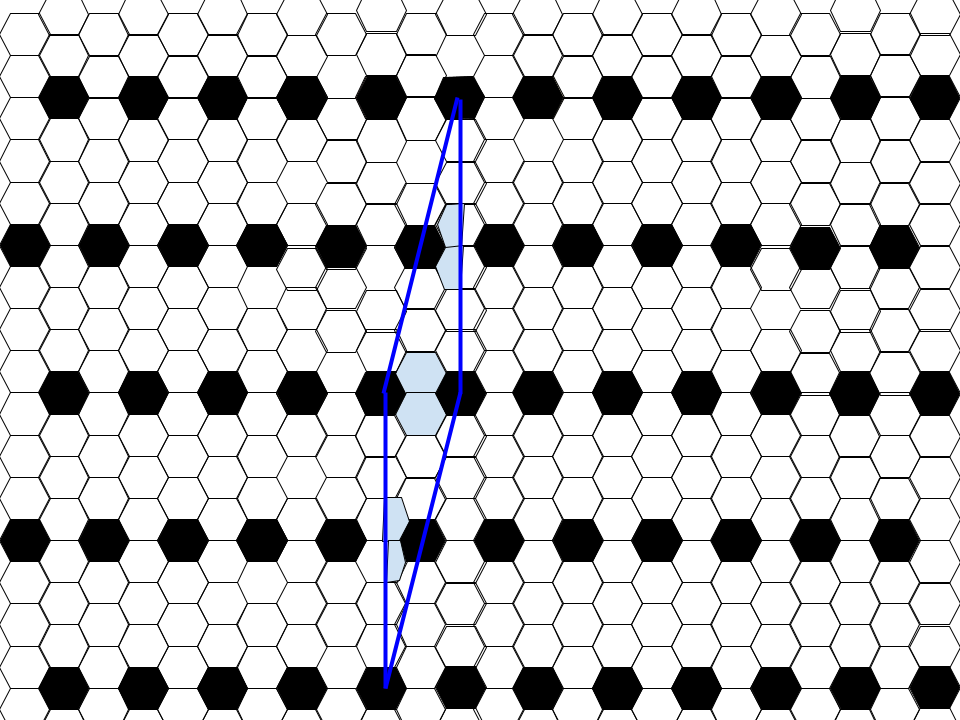}
    \caption{Tight (3,6)-fullerene $(6, 0, 3)$ with $2\star2$ symmetry, represented by its hexagonal tiling cover. The  parallelogram surrounds a fundamental domain. Hexagons shaded in blue are each adjacent to two special hexagons on non-opposite and non-adjacent edges, and therefore to two triangles downstairs in the $(3,6)$-fullerene.}
    \label{fig:2star2Tight}
\end{figure}

The following proposition gives conditions for the existence of (3,6)-fullerenes with $\star332$ symmetry: 

\begin{proposition}
There is a (3,6)-fullerene that is not a godseye with $V$ vertices and $\star 332$  symmetry if and only if $\frac{V}{4}$ is either a perfect square or 3 times a perfect square. 
\end{proposition}

\begin{proof}
    By Proposition 7.2 of \cite{green2025enumerate}, 
    a (3,6)-fullerene has 3-fold rotational symmetry and mirror symmetry if and only if its signature is either of the form $(m-1, m-1, 0)$ or else of the form $(3m - 1, m- 1, m)$ for some $m \geq 1$. Since $\frac{V}{4} = (s+1)(b+1)$ by Equation~\ref{eqn-vCount}, there is a (3,6)-fullerene with $V$ vertices that has $\star 332$ symmetry if and only if $\frac{V}{4} = m^2$ or $\frac{V}{4} = 3m^2$. None of these (3,6)-fullerenes are godseyes since these signatures are not of the form $(n, 0, 0)$, $(n, 0, n)$, or $(0, n, 0)$ for $n \geq 1$.
\end{proof}

Next, we give conditions for the existence of (3,6)-fullerenes with $332$ symmetry. 
\begin{proposition}
There is a (3,6)-fullerene that is not a godseye with $V$ vertices  that has $332$  symmetry if and only if, in the prime factorization of $\frac{V}{4}$, all prime factors congruent to $2 \pmod 3$ have even exponents and there is at least one prime factor congruent to $1 \pmod 3$. 
\end{proposition}

\begin{proof}
Note that godseyes do not have 3-fold rotational symmetry, since their three signatures $(n, 0, 0)$, $(n, 0, n)$ and $(0, n, 0)$ for $n \geq 1$ are not identical. So the condition that the (3,6)-fullerene is not a godseye, which is included to agree with the restriction to convex polyhedra in \cite{deza2005zigzag}, is not actually needed.

Suppose $\frac{V}{4}$ factors as $ \displaystyle  2^x 3^y  \prod_{i = 1}^z p_i^{k_i} \prod_{j = 1}^{w} q_j^{\ell_j}$,  where $x$, $y$, $z$, and $w$ are all $\geq 0$, the $p_i$ and $q_j$ are primes, $p_i \equiv 1 \pmod 3$ $\forall i$, $q_j \equiv 2 \pmod 3$ and $q_j > 2$ $\forall j$. By Proposition~\ref{prop-count}, the number of (3,6)-fullerenes with $V$ vertices and $332$ symmetry is equal to 
\begin{enumerate}
\item [(a)] $\displaystyle \prod_{i = 1}^z (k_i + 1) - 1$ if all $k_i$ are even, all $\ell_j$ are even, and $x$ is even
\item[(b)] $\displaystyle \prod_{i = 1}^z (k_i + 1)$ if all $\ell_j$ are even and $x$ is even, but at least one $k_i$ is odd
\item[(c)] 0 if $x$ is odd or any $\ell_j$ is odd
\end{enumerate}
These numbers can equal 0 only in case (c), when there is a prime congruent to 2 mod 3 with an odd exponent, or in case (a) if all $k_i$ are zero. Therefore, the numbers are greater than zero in all other cases, namely when all  prime factors congruent to 2 mod 3 have even exponents and there is at least one prime factor congruent to 1 mod 3. 
\end{proof}

Let $\tau(V)$ represent the number of (3,6)-fullerenes with $V$ vertices with 222 symmetry. It is straighforward to give a formula for $\tau(V)$ in terms of previously discussed quantities.

\begin{proposition} The number of (3,6)-fullerenes with $V$ vertices with 222 symmetry is given by the formula
$$\tau(V) =  \frac{1}{3} \sigma(V) - \frac{1}{3} \delta(V) - \mu(V) + \nu(V),$$ 
where 
\begin{itemize}
    \item $\sigma(V)$ is the number of signatures of (3,6)-fullerenes with $V$ vertices, 
    \item $\delta(V)$ is the number of (3,6)-fullerenes with $V$ vertices with 3-fold rotational symmetry, i.e. $\star 332$ or $332$ symmetry
    \item $\mu(V)$ is the number of (3,6)-fullerenes with $V$ vertices with mirror symmetry, i.e. $\star 332$, $2\star 2$, or $\star 222$ symmetry,
    \item $\nu(V)$ is the number of (3,6)-fullerenes with $V$ vertices with both 3-fold rotational symmetry and mirror symmetry, i.e., with $\star 332$ symmetry. 
    \end{itemize}
\label{prop-eta}
\end{proposition}

\begin{proof}
Let $\alpha(V)$ be the number of (3,6)-fullerenes with V vertices. 
From Proposition~\ref{prop-222}, all (3,6)-fullerenes that do not have 3-fold rotational symmetry or mirror symmetry must have 222 symmetry. So  $\tau(V) = \alpha(V) - \delta(V) - \mu(V) + \nu(V)$. 

From Theorem 5.2 of  \cite{green2025enumerate}, $\alpha(V) = \frac{1}{3} \sigma(V) + \frac{2}{3} \delta(V)$,  so $\tau(V) = \frac{1}{3} \sigma(V) + \frac{2}{3} \delta(V) - \delta(V) - \mu(V) + \nu(V) = \frac{1}{3} \sigma(V) - \frac{1}{3} \delta(V) - \mu(V) + \nu(V)$.

 \end{proof}

 Explicit formulas for $\nu(V)$ and $\mu(V)$, in terms of the prime factorization of $\frac{V}{4}$, are given in parts (a) and (d) of Proposition~\ref{prop-count}. An explicit formula for $\delta(V)$ is given by the sum of the formulas of parts (a) and (e) of Proposition~\ref{prop-count}. % Theorem 5.1, Proposition 7.3, and Proposition 7.2 of \cite{green2025enumerate}. 
 An explicit formula for $\sigma(V)$ is given in Equation 9 of \cite{green2025enumerate}. For completeness, this formula is copied below.  
 
 Suppose $\frac{V}{4}$ factors  as $\displaystyle \prod_{i = 1}^z p_i^{k_i}$, where the $p_i$ are distinct primes.
 \begin{equation}
    \displaystyle \sigma(V) = \prod_{i = 1}^z \dfrac{p_i^{k_i + 1} - 1}{p_i - 1}
    \label{eqn-sigma}
\end{equation}

The formulas for $\sigma(V)$, $\delta(V)$, $\mu(V)$, and $\nu(V)$ can be combined with the formula for $\tau(V)$ in Proposition~\ref{prop-eta} to compute the number of $(3,6)$-fullerenes with 222 symmetry.

Next, we will prove the first statement of  Conjecture 5.7 (iii) from \cite{deza2005zigzag}, which can be stated as: 

\begin{proposition}
There is a (3,6)-fullerene that is not a godseye with $V$ vertices and $222$ symmetry if and only if $V \geq 24$ and $V \neq 28, 32$. 

%Original statement:
%    $3_n(D_2)$ exists if and only if  $n \geq 24$ and $n \neq 28, 32$. 
\end{proposition}

\begin{proof}
Note that all godseyes have mirror symmetry, so the condition that the (3,6)-fullerene is not a godseye, which is included to agree with the restriction to convex polyhedra in \cite{deza2005zigzag}, is not actually needed. 
The condition that $V \geq 24$ and $V \neq 28, 32$ is equivalent to the condition that $\frac{V}{4} \notin \{1, 2, 3, 4, 5, 7, 8\}$. 

To see that there are no (3,6)-fullerenes with 222 symmetry with $\frac{V}{4} \in \{1, 2, 3, 4, 5, 7, 8 \}$, it is straightforward to use the formula for $\tau(V)$ in Propsition~\ref{prop-eta} and the formulas for $\sigma(V), \delta(V), \mu(V)$ given above.

For example, for $\frac{V}{4} = 8$, since $8$ factors as $2^3$, $\sigma(32) = \dfrac{2^4 - 1}{2 - 1} = 15$, $\delta(32) = 0$, $\mu(32) = 5$, and $\nu(32) = 0$, so by {Proposition}~\ref{prop-eta}, $\tau(32) = \frac{1}{3} \cdot 15 - \frac{1}{3} \cdot 0 - 5 + 0 = 0$.

%For $V = 28$ and $\frac{V}{4} = 7$, $\sigma(28) = \dfrac{7^2 - 1}{7 - 1} = 8$, $\delta(28) = 2$, $\mu(28) = 2$, and $\nu(=28) = 0$, so $\tau(28) = \frac{1}{3} \cdot 8 - \frac{1}{3} \cdot 2 - 2 + 0 = 0$.

The other cases, with $\frac{V}{4} = 1, 2, 3, 4, 5, 7$, follow similarly. See the counts in Table~\ref{table-countBySymmType}.

To see that there do exist (3,6)-fullerenes with 222 symmetry with $V$ vertices for $V \notin \{1, 2, 3, 4, 5, 7, 8\}$, it is possible to use the formulas in Proposition~\ref{prop-count},  Equation~\ref{eqn-sigma}, and Proposition~\ref{prop-eta} to show that $\tau(V) > 0$. However, a constructive approach is simpler.  

Consider the (3,6)-fullerene with signature $(s, 0, 2)$. We will show that if this (3,6)-fullerene has 3-fold symmetry or mirror symmetry, then $s \in \{ 0, 1, 2, 3, 4, 6, 7 \}$. Consequently, if $s \notin \{ 0, 1, 2, 3, 4, 6, 7 \}$, then $(s, 0, 2)$ has no 3-fold symmetry and no mirror symmetry, and therefore must have 222 symmetry by Proposition~\ref{prop-222}.  Since  $\frac{V}{4} = (s+1)(b+1)$ by Equation~\ref{eqn-vCount} and $b = 0$, $\frac{V}{4} = s+1$. So $s \notin \{ 0, 1, 2, 3, 4, 6, 7 \}$ is equivalent to $\frac{V}{4} \notin \{1, 2, 3,4, 5, 7, 8\}$. Therefore, the (3,6)-fullerene $(s, 0, 2)$ will be an example of a (3,6)-fullerene with 222 symmetry for $\frac{V}{4} \notin \{1, 2, 3,4, 5, 7, 8\}$ and this will prove the proposition.

Suppose the (3,6)-fullerene $(s, 0, 2)$ has 3-fold rotational symmetry. We need to show that $s \in \{ 0, 1, 2, 3, 4, 6, 7 \}$.
By Proposition~3.2 of \cite{green2025enumerate}, it must have coinciding signatures, and by Proposition~4.2 of \cite{green2025enumerate}, $(s, 0, 2)$ is of the form $(cm-1, m-1, gm)$ for some $m \geq 1$, $c \geq 0$, and $g \geq 0$, where $g^2 + g + 1 \equiv 0 \pmod c$. So $m = 1, c = s + 1$, and $g = 2$, and $\displaystyle 2^2 + 2 + 1 \equiv 0 \pmod {s+1}$.  That is, $\displaystyle 7 \equiv 0\pmod{s+1}$, so either $s+1 = 7$ or $s + 1 = 1$.
So $s \in \{0, 1, 2, 3, 4, 6, 7\}$.

Suppose the (3,6)-fullerene $(s, 0, 2)$ has mirror symmetry. We need to show that $s \in \{ 0, 1, 2, 3, 4, 6, 7 \}$. By Lemma~6.2 of \cite{green2025enumerate}, either $(s, 0, 2)$ is self-mirror, or else one of the two equivalent signatures for this (3,6)-fullerene is self-mirror, and the signature $(s, 0, 2)$ has the remaining signature as its mirror signature.

    Case 1: $(s, 0, 2)$ is a self mirror.
Then by definition, the mirror image signature $(s, 0, s - 0 - 2 \pmod {s+1})$ is equal to $(s, 0, 2 \pmod {s+1})$. So $s-2 \equiv 2 \pmod {s+1}$, 
so $s+1 \equiv 5 \pmod {s+1}$.
Therefore, either $s+1 = 5$ or $s+ 1 = 1$. So $s \in \{0, 1, 2, 3, 4, 6, 7\}$.

Case 2: $(s, 0, 2)$ is not self-mirror. In this case, we use Equations~\ref{eqn-s2}-\ref{eqn-f3} to find alternative signatures, where $(s_1, b_1, f_1) = (s, 0, 2)$.

Subcase 2a: $s$ is even. By Equation~\ref{eqn-s2},
$s_2 = j_2(0+1)-1$ where $j_2$ is the smallest positive integer such that $j_2 \cdot 2 \equiv 0 \pmod{s+1}$. Since $s$ is even by assumption, $s+1$ is odd, so $j_2 = s+1$. Hence, $s_2 = (s+1)(0+1)-1 = s$. Since $s_2 = s_1 = s$ and $b_1 = 0$, by  Equation~\ref{eqn-b2}, $b_2 = 0$.
By Equation~\ref{eqn-f2}, 
$f_2 = -p_2(0 + 1) - (0+1)\pmod{s+1} = -p_2 - 1 \pmod{s+1}$, where $p_2$ is the smallest positive integer such that $p_2 \cdot 2 \equiv 1 \pmod {s+1}$. There is only one number less than or equal to $ s+1$ that satisfies this equation, namely $p_2 = \dfrac{s+2}{2}$, since $s$ is even. Therefore, $f_2 \equiv -\dfrac{s+2}{2} - 1  \equiv \dfrac{s}{2} - 1 \pmod {s+1}$. So the alternative signature is $(s, 0, \frac{s}{2} -1)$.

If this alternate signature is self-mirror, then $$(s, 0, s - 0 - (\frac{s}{2} - 1) \pmod{s+1}) = (s, 0, \frac{s}{2} - 1 \pmod{s+1}),$$ so $\frac{s}{2} + 1 \equiv \frac{s}{2} -1 \pmod{s+1}$, so $2\equiv 0 \pmod {s+1}$. It follows that either $s+1 = 2$ or $s+1 = 1$, so $s \in \{0, 1, 2, 3, 4, 6, 7\}$.
If it is not self-mirror, then the two signatures $(s, 0,2)$ and $(s, 0, \frac{s}{2} - 1)$ must be mirror signatures of each other by Lemma~9 of \cite{green2025enumerate}. That is, $$(s, 0, s-0-2 \pmod{s+1}) = (s, 0, \frac{s}{2} - 1 \pmod {s+1}),$$ so $s  - 2 \equiv \frac{s}{2} - 1 \pmod {s+1})$. Simplifying gives $\frac{s}{2} \equiv 1 \pmod {s+1}$, so $s \equiv 2 \pmod {s+1}$, so $-1 \equiv 2 \pmod {s+1}$. If follows that either $s+1 = 3$ or $s+1 = 1$, so $s \in \{0, 1, 2, 3, 4, 6, 7\}$.

Subcase 2b: $s$ is odd. 
In Equation~\ref{eqn-s2}, $j_2$ is the smallest positive integer such that $j_2 \cdot 2 \equiv 0 \pmod {s + 1}$. Since $s+1$ is even, $j_2 = \dfrac{s+1}{2}$. 
Therefore, Equation~\ref{eqn-s2} yields $s_2 = \frac{s+1}{2}(1) - 1 = \frac{s-1}{2}$.
By Equation~\ref{eqn-b2},  
$b_2 = \frac{h-2s_2}{2s_2+2}$, where $h = 2s_1b_1 + 2s_1 + 2b_1 = 2s$. Therefore, $b_2 =  \dfrac{2s-2(\frac{s-1}{2})}{2(\frac{s-1}{2})+2} =1$.

To solve for $f_2$ using Equation~\ref{eqn-f2}, we first find $p_2$, the smallest positive integer such that $p_2\cdot 2 \equiv 2\pmod{s+1}$. Since $s+1$ is even, there are two solutions to $p_2(2) \equiv 2\pmod{s+1}$, namely $p_2 = 1$ and $p_2 = \dfrac{s+1}{2} + 1$. The solution $p_2 = 1$ is the smallest.  Plugging into Equation~\ref{eqn-f2} yields $f_2 = -p_2(b_1 +1) - (b_2+1)\pmod{s_2+1} = -3\pmod{\frac{s+1}{2}}$. From this, we see that the alternate signature for $(s, 0, 2)$ if s is odd is $(\frac{s-1}{2}, 1, -3)$.

If this signature is self-mirror, then 
 $\frac{s-1}{2} - 1 + 3 \equiv-3\pmod{\frac{s+1}{2}}$, so $-1 - 1 + 3 \equiv-3\pmod{\frac{s+1}{2}}$.  So $\frac{s+1}{2} \mid 4$, so $s+1 \mid 8$.  Therefore, $s \in \{0, 1, 2, 3, 4, 6, 7\}$. 

 Alternatively, if this signature is not self-mirror, then the two signatures $(s, 0, 2)$ and $(\frac{s-1}{2}, 1, -3)$ must be mirror signatures by Lemma~9 of \cite{green2025enumerate}. So $s  = \frac{s-1}{2}$, which is not possible. 

 All possible cases where $(s, 0, 2)$ has 3-fold rotational symmetry or mirror symmetry force $s \in \{0, 1, 2, 3, 4, 6, 7\}$. Therefore, if $s \notin \{0, 1, 2, 3, 4, 6, 7\}$, i.e. $\frac{V}{4} \notin \{1, 2, 3, 4, 5, 7, 8\}$, then $(s, 0, 2)$ is (3, 6)-fullerene $V$ vertices with 222 symmetry.

\end{proof}

\pagebreak
\section{Symmetry Group Proofs}
\label{sec-symmProof}

In this section, we develop the background on symmetry groups needed to prove Proposition~\ref{prop-symmEquivSummary} and \ref{prop-upstairsDownstairsSymm} from Section~\ref{sec-symmDefs}. Our  methods for proving Proposition~\ref{prop-upstairsDownstairsSymm} follow Martin Bridson and Andr\'e Haefliger \cite{bridson2013metric}.

Let $F$ be the graph of a $(3, 6)$-fullerene and let $f: F \to S^2$ be an embedding into the sphere.   We will say that a homeomorphism $\alpha: S^2 \to S^2$ \emph{preserves} $f(F)$ if $\alpha$ takes vertices of $f(F)$ to vertices of $f(F)$ and edges of $f(F)$ to edges of $f(F)$. We will say that $\alpha$ \emph{fixes} $f(F)$ if for every vertex $v$ in $F$, $\alpha$ takes $f(v)$ to itself and for every edge $e$ in $F$, $\alpha$ takes $f(e)$ to itself. 
As in Section~\ref{sec-symmDefs}, for a given embedding in the sphere $f:F \to S^2$, let $Isom(S^2, F, f)$ be the group of isometries of the sphere that preserve $f(F)$. Recall from Section~\ref{sec-symmDefs} that an embedding $f: F \to S^2$ is called \emph{maximally symmetric} if there is no embedding $f': F \to S^2$ such that $Isom(S^2, F, f)$ is isomorphic to a proper subset of $Isom(S^2, F, f')$.

It will be useful to consider the symmetries of a $(3, 6)$-fullerene that is represented as the orbifold  quotient of a hexagonal tiling of the plane, as described in Section~\ref{sec-background}. As in that section, consider a regular hexagonal grid on the plane, with a superimposed parallelogram grid whose vertices lie in the center of hexagons. Let $\Gamma$ be the group of isometries of the plane generated by $180^\circ$ rotations around the vertices of parallelogram grid.  Let $Q = \Gamma \backslash \mathbb{R}^2$ be the quotient orbifold with the induced metric described in Section~\ref{sec-background}. Let $i: F \to Q$ be the embedding of $F$ in $Q$. Let $Isom(Q, F, i)$ represent the isometries of $Q$ that preserve  $i(F)$. Since the embedding $i$ is uniquely defined by the quotient structure (up to automorphisms of $F$), we will simplify notation by suppressing reference to the embedding function $i$. Thus, we will  write $F \subset Q$ and use the notation $Isom(Q)$ to mean $Isom(Q, F, i)$. 

For the graph $F$ of a $(3,6)$-fullerene, let $Aut(F)$ be its graph automorphism group. For any embedding $f:F \to S^2$, there is a natural homomorphism from $Isom(S^2, F, f)$ to $Aut(F)$ induced by restricting the isometry to the graph.  For 3-connected $(3, 6)$-fullerenes, it is straightforward to prove that for a maximally symmetric embedding, $f_0: F \to S^2$, this homomorphism is an isomorphism. Surjectivity follows from a theorem of Mani \cite{MANI1971}, which states that for any 3-connected, planar graph $G$,
there is a convex polyhedron with edge and vertex graph $G$, such that every element of $Aut(G)$ is induced by an isometry of the polyhedron. Details are omitted since this isomorphism is not needed to prove the propositions in this section.

All $(3,6)$-fullerenes are 3-connected with the exception of godseye, described in Section~\ref{sec-background} and illustrated in  Figure~\ref{fig-godseye}. For a maximally symmetric embedding of a godseye $f_0: F \to S^2$, the homomorphism $Isom(S^2, F, f_0) \to Aut(F)$ is not surjective, because there are Whitney flips that give  graph automorphisms that do not correspond to a homeomorphism of the sphere, and therefore can't correspond to isometries. However, for a maximally symmetric embedding $f_0: F \to S^2$, it is  possible to show that $Isom(S^2, F, f_0)$ is isomorphic to $Isom(Q)$, no matter whether $F$ is 2-connected or 3-connected. This is accomplished in the following sequence of lemmas, leading up to the proof of Proposition~3.1.

For any $(3,6)$-fullerene $F$ with embedding $f:F \to S^2$, let 
$MCG(S^2, F, f)$ represent the mapping class group of $S^2$ relative to $f(F)$, defined as the group of homeomorphisms of $S^2$ that preserve the graph $f(F)$, modulo the equivalence relation in which two homeomorphisms $h_1$ and $h_2$ are equivalent if there is an isotopy between them that preserves $f(F)$.

\begin{lemma} Let $F$ be the graph of a $(3,6)$-fullerene.
For any embedding $f:F \to S^2$, let $\pi_*: Isom(S^2, F, f) \to MCG(S^2, F, f)$ be the homomorphism defined by sending an isometry to its equivalence class of homeomorphisms in the mapping class group. This homomorphism $\pi_*$ is injective. 
\label{lem-IsomInjMCG}
\end{lemma}

\begin{proof}
    Suppose that $\alpha$ is an isometry of $S^2$ that fixes the graph $f(F)$. If $\alpha$ is sent to the identity element of $MCG(S^2, F, f)$, then $\alpha$ is isotopic to the identity map through an isotopy that preserves the graph $f(F)$. This isotopy must preserve the vertices of $f(F)$, and therefore, the original map $\alpha$ must fix each vertex.  Since $F$ is a cubic graph, there must be at least three vertices of $f(F)$ that are not on a great circle. Any isometry of $S^2$ that fixes 3 points that are not on a great circle is the identity. Therefore, $\alpha$ must the the identity, so the homomorphism $\pi_*: Isom(S^2, F, f) \to MCG(S^2, F, f)$ is injective. 
\end{proof}

\begin{lemma}
Let $F$ be the graph of a $(3,6)$-fullerene.  For any two embeddings $f_1: F \to S^2$ and $f_2:F \to S^2$, $MCG(S^2, F, f_1) \cong MCG(S^2, F, f_2)$.
\label{lem-allMCG}
\end{lemma}

\begin{proof}
Let $f_1: F \to S^2$ and $f_2: F \to S^2$ be two embeddings of $F$ in $S^2$. If $F$ is 3-connected, then by Whitney's uniqueness theorem \cite{whitneyCongruentGraphs1932}, there is a homeomorphism $h: S^2 \to S^2$ that takes vertices and edges of $f_1(F)$ to vertices and edges of $f_2(F)$. If $F$ is, instead, a 2-connected godseye, then by Whitney’s 2-isomorphism theorem \cite{whitneyClassificationOfGraphs1933}, there is a sequence of Whitney flips along pairs of separating vertices that transforms the embedding $f_1$ to $f_2$. For a godseye, these Whitney flips do not change the face structure of the embedding. So there still exists a homeomorphism $h: S^2 \to S^2$ that takes vertices and edges of $f_1(F)$ to vertices and edges of $f_2(F)$, respectively. Conjugation by this homeomoprhism induces an isomorphism from $MCG(S^2, F, f_1)$ to $MCG(S^2, F, f_2)$. 
\end{proof}

\begin{lemma}
     Let $Q$ be the quotient representation of a $(3, 6)$-fullerene with graph $F$. 
Then $Isom(Q) \cong MCG(S^2, F, f)$ for any embedding $f: F \to S^2$.  
     \label{lem-QIsoMCG1}
\end{lemma}

\begin{proof}
$Q$ is homeomorphic to $S^2$ \cite{green2024polyhedra}. 
Let $g: Q \to S^2$ be a homoeomorphism, so that $g$ induces an embedding of $F$ in $S^2$. 
Let  $g_*: Isom(Q) \to Homeo(S^2, F, g)$ be  the homomorphism induced by conjugation by $g$, and let  $\pi_*: Homeo(S^2, F, g) \to MCG(S^2, F, g)$ be  the projection homomorphism that takes any homeomorphism to its equivalence class. 
The homomorphism $\pi_* \circ g_*$ is injective, by the following argument. For any $\alpha \in Isom(Q)$, if $g \circ \alpha \circ g^{-1}$ is isotopic to the identity on $S^2$, through an isometry that preserves $g(F)$, then $\alpha$ is isotopic to the identity on $Q$ through an isotopy that preserves $F$. But since the identity fixes each vertex and edge of $F$, this isotopy must fix each vertex and edge of $F$. So $\alpha$ fixes each vertex and edge of $F$. Because $\alpha$ is an isometry, it must pointwise fix each edge of $F$. Since each face of $Q - F$ is either a regular hexagon or a triangular cone made from a regular hexagon, any isometry that pointwise fixes the edges of a face must pointwise fix the entire face. So $\alpha$ is the identity isometry. 

In addition, $\pi_* \circ g_*$ is surjective, by the following argument. If $\beta$ is a homeomorphism representing an element of $MCG(S^2, F, g)$, then $ g^{-1} \circ \beta \circ g $ is a homeomorphism of $Q$ that preserves the graph $F$. Since $Q$ is the quotient of a regular hexagonal tiling, this homeomorphism is isotopic to an isometry $\gamma$, via an isotopy that preserves $F$, using the Alexander trick. Since $g_*(\gamma)$ represents the same element of $MCG(S^2, F, g)$ as $\beta$, the homomorphism $\pi_* \circ g_*$ is surjective. 

Therefore, $\pi_* \circ g_*$ is an isomorphism, and  $Isom(Q) \cong MCG(S^2, F, g)$. If $f: F \to S^2$ is any other embedding, then by Lemma~\ref{lem-allMCG}, $MCG(S^2, F, g) \cong MCG(S^2, F, f)$, so $Isom(Q) \cong MCG(S^2, F, f)$, as wanted. 

\end{proof}

\begin{lemma}
     Let $Q$ be the quotient representation of a $(3, 6)$-fullerene with graph $F$. Then $Isom(Q)$ is a finite group. 
     \label{lem-finite}
\end{lemma}

\begin{proof}
Consider the homomorphism $\pi_0: Isom(Q) \to Aut(F)$ induced by restricting an isometry on $Q$ to the embedded graph $F$. This homomorphism is injective, by the following argument. Any isometry of $Q$ that fixes $F$ must fix all vertices and edges. Since it is an isometry, if must fix all points of edges. Therefore, it must also fix all points of the triangles and hexagons in $Q - F$. 

Since $F$ is a finite graph, $Aut(F)$ is finite. Since $\pi_0: Isom(Q) \to Aut(F)$ is injective, $Isom(Q)$ must be finite as well. 
\end{proof}

\begin{lemma}
    Let $F$ be the graph of a $(3,6)$-fullerene. There exists an embedding $k: F \to S^2$ such that $ Isom(S^2, F, k) \cong MCG(S^2, F, k)$.
    \label{lem-IsoIsoMCG}
\end{lemma}

\begin{proof}
Consider a homeomorphism $g: Q \to S^2$, where $Q$ is the quotient representation of the $(3, 6)$-fullerene. This homeomorphism restricts to an embedding $g: F \to S^2$. The map $g$ induces a homomorphism $g_*: Isom(Q) \to Homeo(S^2, F, g)$ via conjugation. This homomorphism is clearly injective, so it creates an isomorphism from $Isom(Q)$ to its image $g_*(Isom(Q)) \subseteq Homeo(S^2, F, g)$. By Lemma~\ref{lem-finite}, $Isom(Q)$ is finite, so $g_*(Isom(Q))$ is a finite group of homeomorphisms of the sphere. Therefore, it is  conjugate, via a homeomorphism of $S^2$, to a subgroup of  the isometry group of the sphere \cite{kolev2006}. That is, there is a map $h: S^2 \to S^2$ such that $h_*: Homeo(S^2) \to Homeo(S^2)$ defined by conjugation $h_*: \beta \mapsto h \circ \beta \circ h^{-1}$  induces an isomorphism from $g_*(Isom(Q))$ to a subgroup of $Isom(S^2)$. Since $h$ takes vertices and edges of $g(F)$ to vertices and edges of $h(g(F))$, the image of $h_* \circ g_*$ is actually in $Isom(S^2, F, h \circ g)$.  Since $\pi_*: Isom(S^2, F, h \circ g) \to MCG(S^2, F, h \circ g)$ is injective by Lemma~\ref{lem-IsomInjMCG}, $\pi_* \circ h_* \circ g_*$ is an isomorphism from $Isom(Q)$ to a subgroup of $MCG(S^2, F, h \circ g)$. But $Isom(Q) \cong MCG(S^2, F, h \circ g)$ by Lemma~\ref{lem-QIsoMCG1} and is finite by Lemma~\ref{lem-finite}, and a finite group cannot be isomorphic to a proper subgroup of itself. Therefore, $\pi_*: Isom(S^2, F, h \circ g) \to MCG(S^2, F, h \circ g)$ must be surjective and so $Isom(S^2, F, h\circ g) \cong MCG(S^2, F, h \circ g)$.    
This proves the lemma, with $k = h \circ g$.

\end{proof}

We can now prove Proposition~\ref{prop-symmEquivSummary}, restated here for convenience. This proposition shows that the symmetry group of a $(3,6)$-fullerene can be thought of as either the isometry group of the quotient representation or as the isometry group of a maximally symmetric embedding into the sphere. 

\noindent {\bf Proposition 3.1.}
\begin{it}
    Suppose $Q$ is a quotient representation of a $(3,6)$-fullerene with edge and vertex graph $F$. Then $Isom(Q) \cong Isom(S^2, F, f_0)$, where $f_0: F \to S^2$ is any maximally symmetric embedding. 
\end{it}

\begin{proof}
Let $f_0: F \to S^2$ be a maximally symmetric embedding. By Lemma~\ref{lem-IsomInjMCG},  $Isom(S^2, F, f_0)$ is isomorphic to a subgroup of $MCG(S^2, F, f_0)$.  By 
 Lemma~\ref{lem-allMCG},  for any embedding $k: F \to S^2$, $MCG(S^2, F, f_0) \cong MCG(S^2, F, k)$.  By Lemma~\ref{lem-IsoIsoMCG}, there exists an embedding $k: F \to S^2$ such that $MCG(S^2, F, k) \cong Isom(S^2, F, k)$. Therefore,  $Isom(S^2, F, f_0)$ is isomorphic to a subgroup of $Isom(S^2, F, k)$.
Since $k$ is a maximally symmetric embedding, 
 $Isom(S^2, F, f_0) \cong Isom(S^2, F, k)$. By Lemma~\ref{lem-QIsoMCG1}, $Isom(Q) \cong MCG(S^2, F, k)$. \\
 Since $MCG(S^2, F, k) \cong Isom(S^2, F, k) \cong Isom(S^2, F, f_0)$, it follows that $Isom(Q) \cong Isom(S^2, F, f_0)$. 

 \end{proof}

For the remainder of this section, let $W$ be the set of points on the edges and vertices of a hexagonal grid on $\mathbb{R}^2$, and let $Z$ the vertices of a superimposed parallelogram grid, with each vertex in the center of a hexagon.  Let $\text{Isom}(\mathbb{R}^2, Z)$ be the group of isometries of $\mathbb{R}^2$ that  preserve $Z$, and let $\text{Isom}(\mathbb{R}^2, Z \cup W) \subseteq \text{Isom}(\mathbb{R}^2, Z)$ be the group of isometries that preserve both $Z$ and $W$. Let $\Gamma$ be a subgroup of $\text{Isom}(\mathbb{R}^2, Z \cup W)$ generated by $180^\circ$ rotations around points in $Z$. Let $Q = \Gamma \backslash \mathbb{R}^2$ and let $q: \mathbb{R}^2 \to Q$ be the quotient map. The points $q(Z)$ are called \emph{singular points}.

\begin{lemma}
The subgroup $\Gamma$ is normal in $\text{Isom}(\mathbb{R}^2, Z)$. 
\label{lem-normal}
\end{lemma}

\begin{proof}
For any isometry $g: \mathbb{R}^2 \rightarrow \mathbb{R}^2$ that preserves $Z$, and for any $\gamma \in \Gamma$ that is a $180^\circ$ rotation around a point $x \in Z$, the composition $g \circ \gamma \circ g^{-1}$ is a $180^\circ$ rotation that takes $g(x)$ to itself, and $g(x) \in Z$.  Since $\Gamma$ is generated by $180^\circ$ rotations around points of $Z$, $\Gamma$ is normal in $\text{Isom}(\mathbb{R}^2, Z)$.
\end{proof}

 \begin{lemma}
Suppose that $f$ is a homeomorphism of $Q$ that preserves $q(Z)$. Suppose $x_0$ is a point in $\mathbb{R}^2$ that does not project to a singular point, and let $y_0 \in \mathbb{R}^2$ be a point in $ q^{-1}(f(q(x_0)))$. Then there is a unique lift $\tilde f:R^2 \rightarrow R^2$ that preserves $Z$ such that $q \circ \tilde f = f \circ q$ and such that $\tilde f(x_0) = y_0$. 

    \[
\begin{tikzcd}
\mathbb{R}^2 \arrow[r, "\tilde f"] \arrow[d, "q"'] 
& \mathbb{R}^2 \arrow[d, "q"] \\
 Q \arrow[r, "f"] 
&  Q
\end{tikzcd}
\]

Suppose, instead, that $x_0$ is a point in $\mathbb{R}^2$ that does project to a singular point,  and $y_0 \in \mathbb{R}^2$ is a point in $ q^{-1}(f(q(x)))$. Then there are exactly two distinct lifts $\tilde f_1:\mathbb{R}^2 \rightarrow \mathbb{R}^2$ and $\tilde f_2:\mathbb{R}^2 \rightarrow \mathbb{R}^2$ that preserve $Z$ such that $q \circ \tilde f_i = f \circ q$  for $i = 1, 2$ and such that $\tilde f(x_0) = y_0$.

The original map $f$ is an isometry of $Q$ if and only if $\tilde f$ is an isometry of $\mathbb{R}^2$.

The original map $f$ preserves $q(W)$ if and only if $\tilde f$ preserves $W$.
\label{lem-lift}
\end{lemma}

\begin{proof} Suppose first that $x_0 \in \mathbb{R}^2$ does not project to a singular point.
Let $X$ be the subset of $Q $ formed by removing singular points: $X = Q - q(Z)$. Let $\tilde X$ be the space $R^2 -  Z$. Note that $q: \tilde X \to X$ is a covering space. Since $Q$ has four singular points corresponding to the four triangles of the $(3, 6)$-fullerene, the fundamental group $\pi_1(X)$ is a free group on three generators, and  $q_\star(\pi_1(\tilde X))$ is the subgroup generated by squares of these generators, where $q_*$ is the homomorphism of the fundamental group induced by $q$. Because $f$ is a homeomorphism that takes singular points to singular points, $f_\star \circ q_\star(\pi_1(\tilde X))$ is also the subgroup generated by squares of generators. Therefore, $(f\circ q)_\star(\pi_1(\tilde X))$ is contained in $q_\star(\pi_1(\tilde X))$, and so the map $f$, restricted to $X$, lifts to a unique map $\tilde f: \tilde X \to \tilde X$ such that $q \circ \tilde f = f \circ q$ and such that $\tilde f(x_0) = y_0$. The map $\tilde f$ can be extended to a map on all of $\mathbb{R}^2$ that preserves $Z$ by sending each point of $Z$ to the unique point of $Z$ that preserves continuity.

If, instead, $x_0 \in \mathbb{R}^2$ does project to a singular point, then $f(q(x_0))$ is also a singular point since $f$ takes singular points to singular points. Consider a small disk $U$ around $x_0$ and a point $x_1 \in U$. Draw a path $\alpha$ from $x_1$ to $x_0$ within $U$. Lift the path $f(q(\alpha))$ to a path $\beta$ that ends in $y_0$. Let $y_1$ be the initial point of the path $\beta$. Then $y_1 \in q^{-1}( f(q(x_1)))$, and $x_1$ does not project to a singular point, so the previous argument shows that there is a lift $\tilde f: \mathbb{R}^2 \to \mathbb{R}^2$ with $q \circ \tilde f = f \circ q$ and $\tilde f(x_1) = y_1$. Since $\tilde f(\alpha)$ lifts $f(q(\alpha))$, by uniqueness of lifts of paths from $X$ to $\tilde X$, $\tilde f(\alpha) = \beta$ and therefore $\tilde f(x_0) = y_0$ in $\mathbb{R}^2$. 

Let $\gamma \in \Gamma$ be a $180^\circ$ rotation around the point $y_0$. Then $\gamma \circ \tilde f$ is also a lift of $f$, since $q(\gamma(\tilde f)) = q(\tilde f) = f \circ q$. Therefore, $\tilde f$ and $\gamma \circ \tilde f$ are two distinct lifts of $f$ that take $x_0$ to $y_0$.

There are only two possible lifts of $f$ that take $x_0$ to $y_0$, since if $q \circ \tilde{f_1} = q\circ \tilde{f_2}$ and $\tilde{f_1}(x_0) = \tilde{f_2}(x_0) = y_0$, then $q  \circ \tilde{f_1}  \circ \tilde{f_2}^{-1} = f \circ q \circ \tilde{f_2}^{-1} = f \circ f^{-1} \circ q = q$, so $\tilde{f_1} \circ \tilde{f_2}^{-1} = \gamma$ for some $\gamma \in \Gamma$. Since $\tilde{f_1}(x_0) = \tilde{f_2}(x_0) = y_0$, $\gamma(y_0) = \tilde{f_1} \circ \tilde{f_2}^{-1} (y_0) = y_0$, so $\gamma$ must either be the identity or a $180^\circ$ rotation around $y_0$. So there are only two possible lifts.

Suppose that $\tilde{f}$ is an isometry. Using the definition of distance and the fact that $\Gamma$ is normal in  $\text{Isom}(\mathbb{R}^2, Z)$, we can see that $f$ is an isometry: for any $x, y \in \mathbb{R}^2$,  
\begin{align*}
d(f(q(x)), f (q(y))) & = d(q (\tilde{f}(x)), q(\tilde{f}(y))) \\
& = d(\Gamma.\tilde{f}(x), \Gamma.\tilde{f}(y)) \\
& = \inf_{\gamma \in \Gamma} d(\tilde{f}(x), \gamma \circ \tilde{f}(y)) \\
& = \inf_{\gamma \in \Gamma} d(\tilde{f}(x), \tilde{f}\circ \gamma(y)) \\
& = \inf_{\gamma \in \Gamma}d(x, \gamma(y)) = d(\Gamma.x, \Gamma.y) = d(q(x), q(y)) \\
\end{align*}
Conversely, if  $f$ is an isometry, then since $q$ is a local isometry away from singular points, $\tilde f$ will be an isometry on $\tilde X$ and therefore on all of $\mathbb{R}^2$.

Note that $q^{-1}(q(W)) = W$, since $x \in q^{-1}(q(W))$ implies that $\gamma(x) \in W$ for some $\gamma \in \Gamma$, so $x \in W$ also. Therefore, since $q \circ \tilde f = f \circ q$, if the map $f$ preserves $q(W)$, then $\tilde f$ preserves $W$.  Conversely, if $\tilde f: \mathbb{R}^2 \to \mathbb{R}^2$ preserves $W$, then since $f\circ q(W) = q\circ \tilde f(W) = q(W)$, $f$ preserves $q(W)$. 

\end{proof}

\begin{lemma} For $g \in \text{Isom}(\mathbb{R}^2, Z)$, let $\Phi(g): Q \to Q$ be defined by $\Phi(g): \Gamma.x \to \Gamma.(g(x))$. 
 Then  $\Phi$ is a homomorphism from  $\text{Isom}(\mathbb{R}^2, Z)$ to $\text{Isom}(Q, q(Z))$.  $\Phi$   induces an isomorphism between $\text{Isom}(\mathbb{R}^2, Z) /\Gamma$ and $\text{Isom}(Q, q(Z))$ and an isomorphism between $\text{Isom}(\mathbb{R}^2, Z \cup W)/\Gamma$ and  $\text{Isom}(Q, q(Z) \cup q(W))$.
\label{lem-isomorphOfGroups}
    \end{lemma}

    \begin{proof}
Since $\Gamma$ is a normal subgroup of $\text{Isom}(\mathbb{R}^2, Z)$ by Lemma~\ref{lem-normal}, the map $\Phi(g): \Gamma.x \to \Gamma.(g(x))$ is well-defined. Since $g$ preserves $Z$, $\Phi(g)$ preserves $q(Z)$. 
Since $g$ is an isometry, it follows from the definition of distance on $Q$ that $\Phi(g)$ is an isometry:  $d(\Phi(g)(\Gamma.x), \Phi(g)(\Gamma.y)) = d(\Gamma.g(x), \Gamma.g(y)) = \inf_{\gamma \in \Gamma} d(g(x), \gamma \circ g(y)) = \inf_{\gamma \in \Gamma} d(g(x), g\circ \gamma(y)) = \inf_{\gamma \in \Gamma}d(x, \gamma.y) = d(\Gamma.x, \Gamma.y)$, using the fact that $\Gamma$ is normal in $\text{Isom}(\mathbb{R}^2, Z)$. It follows quickly from its definition that the map $\Phi$ is a homomorphism.
  
 To find $\text{Ker}(\Phi)$, suppose that $\Phi(g)$ is the identity map. Fix any point $x_0 \in \mathbb{R}^2$. We have that $g(x_0) = \gamma(x_0)$ for some $\gamma \in \Gamma$. Since $g$ and $\gamma$ are both lifts of the identity map on $\Gamma \backslash \mathbb{R}^2$, if $x_0 \notin Z$, then by Lemma~\ref{lem-lift}, $g = \gamma$. (If $x_0 \in Z$, then $g = \gamma \cdot \gamma'$ for some $\gamma' \in \Gamma$.) So $\text{Ker}(\Phi) \subseteq \Gamma$, and clearly $\Gamma \subseteq \text{Ker}(\Phi)$. So $\text{Ker}(\Phi) = \Gamma$  and the induced map $\text{Isom}(\mathbb{R}^2, Z) /\Gamma \to \text{Isom}(Q, q(Z))$ is injective. 

To see that the map $\Phi$ is surjective, use Lemma~\ref{lem-lift} to lift any isometry of $Q$ that preserves singular points to an isometry of $\mathbb{R}^2$. 

 By Lemma~\ref{lem-lift}, $\Phi(g): Q \to Q$ preserves $q(W)$ if and only if $g: \mathbb{R}^2 \to \mathbb{R}^2$ preserves $W$. Therefore, $\Phi$ takes $ \text{Isom}(\mathbb{R}^2, Z \cup W)$ to $\text{Isom}(Q, q(Z) \cup q(W))$. 
 Since $\Gamma = \text{Ker}(\Phi) \subseteq \text{Isom}(\mathbb{R}^2, Z \cup W)$, $\Phi$ induces an isomorphism between $\text{Isom}(\mathbb{R}^2, Z \cup W)/\Gamma$ and  $\text{Isom}(Q, q(Z) \cup q(W))$.

\end{proof}

Note that for any subgroup $K \subseteq \text{Isom}(\mathbb{R}^2, Z)$ that contains $\Gamma$, the map $\Phi$ induces an isomorphism from $K / \Gamma$ to a subgroup of $\text{Isom}(Q,  q(Z))$. Let $Stab_{\Phi(K)} (\Gamma.x)$ represent the stabilizer of the point $\Gamma.x \in Q$ in this subgroup $\Phi(K)$.

\begin{lemma}
Suppose that $K$ is a subgroup of $\text{Isom}(\mathbb{R}^2, Z)$ that contains $\Gamma$ and suppose $x \in \mathbb{R}^2$. Then $(Stab_K(x) \cdot \Gamma)/\Gamma$ is isomorphic to $Stab_{\Phi(K)} (\Gamma.x)$.
\label{lem-stab}
\end{lemma}

\begin{proof} Consider the isomorphism $\Phi: \text{Isom}(\mathbb{R}^2, Z) /\Gamma \to \text{Isom}(Q, q(Z))$defined in Lemma~\ref{lem-isomorphOfGroups}. Fix a point $x \in \mathbb{R}^2$ and a subgroup $K \subseteq \text{Isom}(\mathbb{R}^2, Z)$ that contains $\Gamma$, and consider the induced map $\Phi: (Stab_{K}(x)\cdot \Gamma) / \Gamma  \to \text{Isom}(Q, q(Z))$. 

First, the image of this induced map lies  within $Stab_{\Phi(K)} (\Gamma.x)$ by the following reasoning. For any $g.\Gamma \in (Stab_{K}(x) \cdot \Gamma) / \Gamma$, where $g \in Stab_{K}(x)$, by definition of $\Phi$, $\Phi(g.\Gamma)$ is the map $\Phi(g)$ defined by $\Phi(g)(\Gamma.x) = \Gamma.g(x)$. Since $g(x) = x$, $\Phi(g)(\Gamma.x) = \Gamma.x$. So $\Phi(g)$ stabilizes $\Gamma.x$. Since $g \in K$, $\Phi(g) \in Stab_{\Phi(K)} (\Gamma.x)$, as wanted.

Next, the induced map surjects onto $Stab_{\Phi(K)} (\Gamma.x)$. If $j$ is an isometry of $Q$ that is $\Phi(k)$ for some $k \in K$  and that stabilizes $\Gamma.x$, then  $j(\Gamma.x) = \Phi(k)(\Gamma.x) = \Gamma.k(x) = \Gamma.x$. So for some $\gamma \in \Gamma$, $\gamma(k(x)) = x$. So $\gamma \cdot k \in Stab_K (x)$, so $k \in \Gamma \cdot Stab_K(x) = Stab_K(x) \cdot \Gamma$.

Finally, the induced map is injective, simply because the original map $\Phi$ has kernel $\Gamma$. 
\end{proof}

To simplify notation, for the following corollary and lemmas, let $H$ be another name for $\text{Isom}(\mathbb{R}^2, Z \cup W)$. 
%NOTE: It is not always true that the isometries that preserve $Z$ automatically preserve $W$. As an example: Let v_1 be the straight up vector from (0, 0) to (0, 1) and v_2 be the SW to NE vector from (0, 0) to (sqrt(3)/2, 1/2). Then 2v_1 + v_2 is perp to 4v_1 - 5v_2 so these form a rectangular grid, which has mirror lines in the middle of rectangles, but those mirror lines are not at 60 degree angles so can’t be mirror lines of hexagonal grid

\begin{corollary}
     $\Phi(H)$ is the set of isometries of $Q$ that preserve $q(W)$. These isometries also preserve $q(Z)$. $(Stab_H(x) \cdot \Gamma)/\Gamma \cong Stab_{\Phi(H)} (\Gamma.x)$.
\label{cor-stabcor}
\end{corollary}

\begin{proof}
    Note that $Isom(Q, q(W) \cup q(Z)) = Isom(Q, q(W))$, since isometries of $Q$ that preserve $q(W)$ must take triangles to triangles, and therefore preserve $q(Z)$. The conclusions now follow from  Lemmas~\ref{lem-isomorphOfGroups} and \ref{lem-stab}.
\end{proof}

\begin{lemma} For any two points $z_1, z_2 \in Z$, $Stab_H(z_1) \cong Stab_H(z_2)$. 
\label{lem-singularPoints}
\end{lemma}
\begin{proof}
    For any points $z_1, z_2 \in Z$, there is an isometry $h \in H$ that takes $z_1$ to $z_2$. The homomorphism from $Stab_H(z_1)$ to $Stab_H(z_2)$ given by $g \mapsto h \circ g \circ h^{-1}$ is an isomorphism since its inverse is given by $k \mapsto h^{-1} \circ k \circ h$. 
\end{proof}

\begin{lemma}
    For any point $x \in \mathbb{R}^2 - Z$, $Stab_{\Phi(H)}(q(x)) \cong Stab_H(x)$. For any point $x \in Z$, $Stab_H(x)$ contains a subgroup $C_2$ of order 2 generated by a $180^\circ$ rotation around $x$, and $Stab_{\Phi(H)}(q(x)) \cong Stab_H(x)/C_2$.
    \label{lem-stabReduction}
\end{lemma}

\begin{proof}
  By Lemma~\ref{lem-stab}, for each $x \in \mathbb{R}^2$,  $Stab_{\Phi(H)}(q(x)) \cong (Stab_H(x)\cdot \Gamma)/\Gamma \cong Stab_H(x)  / 
  (Stab_H(x) \cap \Gamma)$. If $x \notin Z$, then $Stab_H(x) \cap \Gamma = \emptyset$, so $Stab_{\Phi(H)}(q(x)) \cong Stab_H(x)$. If $x \in Z$, then $Stab_H(x) \cap \Gamma$ consists of the $180^\circ$ rotation around $x$ together with the identity map, so $ Stab_{\Phi(H)}(q(x)) \cong Stab_H(x)/C_2$.
\end{proof}

In the following proposition, $C_n$ represents the cyclic group of order $n$ generated by a rotation and $D_m$ represents the dihedral group with $2m$ elements generated by a rotation of order $m$ and a reflection. As before, $H$ represents the subgroup of $Isom(\mathbb{R}^2)$ that fixes both $Z$ and $W$.

\begin{proposition}
Suppose $H$ has wallpaper symmetry type $n_1 n_2 \cdots n_k \star  m_1 m_2 \cdots m_\ell$, in orbifold notation, where $k, \ell \geq 0$ and $n_i, m_j \geq 2$ for $1 \leq i \leq k$ and $1 \leq j \leq \ell$. Let $z$ be a point in $Z$.

    \begin{enumerate}
        \item If $Stab_H(z) = C_{n_i}$ with $n_i > 2$ then $\Phi(H)$ has spherical symmetry type \[n_1 n_2 \cdots n_{i-1} \frac{n_{i}}{2} n_{i+1}  \cdots n_k \star  m_1 \cdots m_\ell \]
        \item If $Stab_H(z) = D_{m_j}$ with $m_j > 2$, then $\Phi(H)$ has spherical symmetry type \[n_1 n_2  \cdots n_k \star  m_1 \cdots m_{j-1} \frac{m_j}{2} m_{j+1} \cdots m_\ell \] .
        \item If $Stab_H(z) = C_{n_i}$ with $n_i = 2$ then $\Phi(H)$ has spherical symmetry type \[n_1 n_2 \cdots n_{i-1} n_{i+1}  \cdots n_k \star  m_1 \cdots m_\ell \]
        \item If $Stab_H(z) = D_{m_j}$ with $m_j = 2$, then $\Phi(H)$ has spherical symmetry type \[n_1 n_2  \cdots n_k \star  m_1 \cdots m_{j-1} m_{j+1} \cdots m_\ell \] .
    \end{enumerate}
\label{lem-generalReduction}
    \end{proposition}

    \begin{proof}
    Consider the symmetry type $n_1 n_2 \cdots n_k \star  m_1 m_2 \cdots m_\ell$. From \cite{conway2016symmetries}, for each number $n_i$ before the $\star $, there is a point $x_i \in \mathbb{R}^2$ with $Stab_H(x_i) = C_{n_i}$. For each $m_j$ after the star, there is a point $y_j \in \mathbb{R}^2$ with $Stab_H(y_j) = D_n$. The points $x_1, x_2, \cdots, x_k, y_1, y_2, \cdots,  y_\ell$ all lie in distinct orbits of $H$, and all orbits of points with non-trivial stabilizers are represented in this list. For each $x_i \notin Z$, by Lemma~\ref{lem-stabReduction},  $Stab_{\Phi(H)}(q(x_i)) = Stab_H(x_i)  =  C_{n_i}$. Likewise, for each $y_j \notin Z$, $Stab_{\Phi(H)}(q(y_j)) = Stab_H(y_j) = D_{m_i}$.
    
    Exactly one point among $x_1, x_2 \cdots , x_k, y_1, y_2, \cdots , y_\ell$ is in $Z$, by the following reasoning.  For any $z \in Z$, $Stab_H(z)$ is non-trivial since it contains the $180^\circ$ rotation around $z$, and $z$ is not in the orbit of any $x \notin Z$ since $H$ preserves $Z$. So at least one point of $Z$ must be among the representatives $x_1, \cdots, x_k, y_1, \cdots, y_\ell$ is in $Z$. Since all of $Z$ is in the same orbit of $H$, no more than one representative from $Z$ is listed. Call this point $v$.  
    
    By Lemma~\ref{lem-stabReduction}, $Stab_{\Phi(H)}(q(v))$ is isomorphic to either $C_{n_i}/C_2 = C_{n_i/2}$ or $D_{m_j}/C_2 = D_{m_j/2}$, depending on whether $Stab_H(v)$ is cyclic or dihedral, using the fact that the $C_2$ subgroup of $D_{m_j}$ is generated by $180^\circ$ rotation rather than reflection. Note that $q(x_1), q(x_2), \cdots , q(x_k),$ $ q(y_1), q(y_2), \cdots , q(y_\ell)$ all lie in different orbits of $\Phi(H)$ by Lemma~\ref{lem-lift}, since $x_1, x_2, \cdots, x_k,$ $ y_1, y_2, \cdots, y_\ell$ all lie in different orbits of $H$. The conclusion of the proposition follows. 
    
    \end{proof}

We can now prove Proposition~\ref{prop-upstairsDownstairsSymm}, restated here for convenience. Note that the symmetry type of the hexagonal cover of a $(3,6)$-fullerene refers to the set of isometries of $\mathbb{R}^2$ that preserve not only the hexagonal tiling but also the vertices of the superimposed parallelogram grid that correspond to the centers of special hexagons. This symmetry group is $Isom(\mathbb{R}^2, Z \cup W)$. The symmetry type of the $(3, 6)$-fullerene refers to the isometry group of the quotient space $Q$ that preserves $q(W)$, and automatically preserves $q(Z)$, since isometries of $Q$ must take triangles to triangles. This symmetry group is $\text{Isom}(Q, q(Z) \cup q(W))$.

\vspace{0.5 cm}

\noindent {\bf Proposition 3.2.}
\begin{it}
\begin{enumerate}

   \item A (3,6)-fullerene has $\star 332$ symmetry if its hexagonal tiling cover has $\star 632$ symmetry.

    \item A (3,6)-fullerene has $332$ symmetry if its hexagonal tiling cover has $632$ symmetry.

    \item A (3,6)-fullerene has $2\star 2$ symmetry if its hexagonal tiling cover has $2 \star 22$ symmetry.
    
    \item A (3,6)-fullerene has $\star 222$ symmetry if its hexagonal tiling cover has $\star 2222$ symmetry.

       \item A (3,6)-fullerene has $222$ symmetry if its hexagonal tiling cover has $2222$ symmetry.

\end{enumerate}
\end{it}

\begin{proof}
As in other proofs in this section, let $H$ be the isometry group of $\mathbb{R}^2$ that preserves $Z$ and $W$, i.e. the symmetry group of the hexagonal cover of the fullerene. So by Corollary~\ref{cor-stabcor}, $\Phi(H)$ is the set of isometries of the $(3, 6)$-fullerene $Q$ that preserve $q(W)$, i.e. the symmetry group of the $(3,6)$-fullerene. 

(1): Suppose that the hexagonal tiling cover has $\star 632$ symmetry. Recall that $q(W)$ consists of four triangles and some number of hexagons \cite{green2024polyhedra}. So $\Phi(H)$ must permute the four triangles, which means that  $\phi(H)$ does not contain any elements of order 6. Therefore, if $x \in \mathbb{R}^2$ is one of the points with $Stab_H(x) = D_6$, then $x$ must be in $Z$, since otherwise by Lemma~\ref{lem-stabReduction}, $Stab_{\Phi(H)}(q(x)) = Stab_H(x)  = D_6$ which is impossible. It now follows from Lemma~\ref{lem-generalReduction} that $\Phi(H)$ has symmetry type $\star 332$.

(2): Suppose that the hexagonal tiling cover has $632$ symmetry. As in the proof of part (1), if $x \in \mathbb{R}^2$ is one of the points with $Stab_H(x) \cong C_6$, then $x$ must be in $Z$. By Lemma~\ref{lem-generalReduction}, $\Phi(H)$ has symmetry type $332$. 

(3):   If $H$ has $2\star 22$ symmetry, then it must have mirror symmetry. By Proposition~6.1 of \cite{green2025enumerate}, there must be a mirror line that ``bisects spines" and therefore goes through points of $Z$. It follows that for points $z \in Z$, $Stab_H(z)$ must be $D_2$ and not $C_2$. Therefore, by Lemma~\ref{lem-generalReduction}, $\Phi(H)$ has symmetry type $2\star 2$. 

(4): Suppose that the hexagonal tiling cover has $\star 2222$ symmetry. For any point $z \in Z$, $Stab_H(z)$ must be $D_2$, since the only non-trivial stabilizers in a $\star 2222$ symmetry type are $D_2$, and $Stab_H(z)$ is non-trivial by Lemma~\ref{lem-stabReduction}. By Lemma~\ref{lem-generalReduction}, $\Phi(H)$ has symmetry type $\star 222$.

Proof of (5): Suppose that the hexagonal tiling cover has $2222$ symmetry. For any point $z \in Z$, $Stab_H(z)$ must be $C_2$, since the only non-trivial stabilizers in a 2222 symmetry type are $C_2$, and $Stab_H(z)$ is non-trivial by Lemma~\ref{lem-stabReduction}. By Lemma~\ref{lem-generalReduction}, $\phi(H)$ has symmetry type $222$.
\end{proof}

\bibliographystyle{plain} 
\bibliography{trihexSymmetries}

\end{document}